\documentclass{amsart}

\usepackage{amssymb}
\usepackage[articlein=false,maxbibnames=99,sortcites,style=ext-numeric]{biblatex}
\usepackage{enumitem}
\usepackage{framed}
\usepackage{footnote}
\usepackage{graphicx}
\usepackage{mleftright}
\usepackage[subrefformat=parens]{subcaption}

\usepackage{thmtools}

\usepackage[
    colorlinks,
    pageanchor,
    pdfusetitle,
    allcolors=blue,
]{hyperref}

\usepackage[nameinlink,capitalize]{cleveref} 
\crefname{equation}{}{}  
\crefname{figure}{Figure}{Figures}

\renewcommand{\left}{\mleft}
\renewcommand{\right}{\mright}

\newcommand{\ch}{\mathop{\mathfrak{ch}}}
\newcommand{\Ch}{\operatorname{Ch}}
\newcommand{\CK}{\mathcal H}
\newcommand{\DiffGp}{\mathfrak D}
\newcommand{\CKE}{\widetilde{\mathcal H}}
\newcommand{\counit}{\varepsilon}
\newcommand{\FdB}{\mathsf{FdB}}

\newcommand{\KK}{\mathbb{K}}

\newcommand{\Aut}{\mathrm{Aut}}

\newcommand{\Cycles}{\mathrm{Z}}

\newcommand{\Forests}{\mathcal F} 
\newcommand{\Homol}{\mathrm{H}}
\newcommand{\id}{\mathrm{id}}

\newcommand{\lin}{\operatorname{lin}}
\newcommand{\linex}{e}
\newcommand{\mel}{\operatorname{mel}}

\newcommand{\od}{\operatorname{od}}
\newcommand{\odv}{\operatorname{\mathbf{od}}}

\newcommand{\Posets}{\mathcal{P}}

\newcommand{\rank}{\operatorname{rk}}
\newcommand{\rankv}{\operatorname{\mathbf{rk}}}

\newcommand{\rt}{\operatorname{rt}}

\newcommand{\RioGp}{\mathfrak{R}}
\newcommand{\RioHopf}{\mathsf{Rio}}

\newcommand{\Trees}{\mathcal T}
\newcommand{\Tub}{\mathrm{Tub}}

\newcommand{\leftact}{\mathbin{\rightharpoonup}}
\newcommand{\rightact}{\mathbin{\leftharpoonup}}

\declaretheoremstyle[
    preheadhook={\setlength{\OuterFrameSep}{0pt} \savenotes \begin{leftbar}},
    postfoothook={\end{leftbar} \spewnotes},
]{niceremark}

\declaretheorem[parent=section]{theorem}
\declaretheorem[sibling=theorem]{conjecture,corollary,lemma,proposition}
\declaretheorem[style=niceremark,sibling=theorem]{remark,example}

\DeclareFieldFormat[article]{issuedate}{#1}
\DeclareFieldFormat[article]{number}{(#1)}
\DeclareFieldFormat[article]{titlecase:title}{\MakeSentenceCase*{#1}}
\DeclareFieldFormat[incollection]{titlecase:title}{\MakeSentenceCase*{#1}}
\DeclareFieldFormat[thesis]{titlecase:title}{\MakeSentenceCase*{#1}}

\title[Algebraic structure of DSEs with multiple insertion places]{The algebraic structure of Dyson--Schwinger equations with multiple insertion places}
\author{Nicholas Olson-Harris and Karen Yeats}
\thanks{KY is supported by an NSERC Discovery grant and by the Canada Research Chairs program.  KY thanks Paul Balduf for useful conversations. NOH thanks Lukas Nabergall for sharing code and for many useful discussions in the early stages of this work.  Thanks also to the referee for their suggestions.}

\newcommand{\Chains}{\mathrm{C}}
\newcommand{\Int}{\mathcal{I}}
\newcommand{\actoncocycle}{\circledast}
\newcommand{\Nat}{\mathbb{Z}_{\geq 0}}
\newcommand{\Natp}{\mathbb{Z}_{\geq 1}}

\begin{document}
\maketitle

\begin{abstract}
We give combinatorially controlled series solutions to Dyson--Schwinger equations with multiple insertion places using tubings of rooted trees and investigate the algebraic relation between such solutions and the renormalization group equation.
\end{abstract}

\section{Introduction}

We explain the algebraic structure of Dyson--Schwinger equations with multiple insertion places, giving nice, combinatorially controlled expansions for the solutions of multiple insertion place Dyson--Schwinger equations in terms of tubings of rooted trees with both vertex and edge decorations.  Along the way we give an algebraic formulation of the renormalization group equation and resolve a conjecture from \cite{nabergall:phd}.  The results are purely algebraic and combinatorial, some of independent interest for researchers in combinatorial Hopf algebras and related areas.  Many of these results first appeared in the PhD thesis of one of us \cite{OH:phd}.

The reader who is not concerned with the physics motivation for these results can skip the rest of this section as well as the first half of \cref{subsec context}.

\medskip

To calculate physical amplitudes in perturbative quantum ﬁeld theory it sufﬁces to understand the renormalized one-particle irreducible (1PI) Green functions in the theory. As series one can obtain the 1PI Green functions by applying renormalized Feynman rules to sums of Feynman graphs. Choosing an external scale parameter $L$ we can think of the Green functions as multivariate series $G(x, L, \theta)$ where $x$ is the perturbative expansion parameter, for us the coupling constant, and the $\theta$ are dimensionless parameters capturing the remaining kinematic dependence.
Dyson--Schwinger equations are coupled integral equations relating the Green functions, the quantum analogues of the equations of motion.

One of us has a long-standing program to find simple combinatorial understandings of the series solutions to Dyson--Schwinger equations.  These first took the form of chord diagram expansions \cite{marie-yeats, cy17, hihn-yeats, cyz19, cy20} and have been recently improved with the new language of tubing expansions by both of us with other authors \cite{tubings}.

These combinatorial solutions to Dyson--Schwinger equations stand in contrast to the Feynman diagram expansions of the Green functions because rather than each Feynman diagram contributing a difficult and algebraically opaque object (its renormalized Feynman integral) each chord diagram or tubing contributes some monomials in the coefficients of the Mellin transforms of the primitive diagrams which drive the Dyson--Schwinger equation and these monomials can be readily read off combinatorially from the chord diagram or tubing.  The tubing expansion has additional benefits: Its origin is algebraic and insightful, coming from the renormalization Hopf algebra and explaining how the original Feynman diagrams map to tubings. It readily generalizes beyond the cases previously solved by chord diagrams \cite{tubings}.  It provides insightful structure for exploring resurgence and non-perturbative properties \cite{bdy}.

All the work so far in this direction has been in the single scale case, which is best exemplified by Green functions for propagator corrections.  In this case we take $L = \log q^2/\mu^2$ where $q$ is the momentum ﬂowing through and $\mu$ is the reference scale for renormalization. In this case there are no further parameters $\theta$.

Prior to the present work, all the work in this direction had a further restriction that only one internal scale of the diagram was considered, that is, we were only working with a single insertion place.  This restriction could be circumvented in the past, so long as all insertions symmetrically were desired, by working with a symmetric insertion place (see \cite{yeats:phd} Section 2.3.3), but this is not a satisfying resolution either physically or mathematically since it does not let us understand the interplay of the different insertion places or give us any control over them.  Herein we resolve this by solving all single scale Dyson--Schwinger equations, allowing multiple distinguished insertion places, using generalized tubing expansions.

Throughout $\KK$ is the underlying field which we take to be of characteristic 0 since this is the case of physical interest.

Specifically, we will be looking to find series solutions $G(x,L)$ to Dyson--Schwinger equations.  For $\mu \in \KK$, the simplest case our Dyson--Schwinger equations look like
\begin{equation}\label{eq simplest dse}
G(x, L) = 1 + \left.xG\left(x, \frac{\partial}{\partial \rho}\right)^\mu (e^{L\rho}-1)F(\rho)\right|_{\rho=0}
\end{equation}
which determines $G(x,L)$ recursively in terms of the coefficients of $F(\rho) = \sum_{i\geq 0}a_i\rho^{i-1}$.  Equivalently, and in the form we'll prefer here
\[
G(x, L) = 1 + x \Lambda(G(x, L)^{\mu})
\]
where $\Lambda$ is a 1-cocycle given by the $a_i$, as will be explained below.  This simple case was solved by a chord diagram expansion \cite{marie-yeats} for negative integer $\mu$, and from there more general forms of Dyson--Schwinger equation solved \cite{hihn-yeats} and further generalizations along with the more conceptual formulation in terms of tubings given \cite{tubings}.  Distinguished insertion places, however, remained open until the present work.  The most general system of Dyson--Schwinger equations that we solve is given in \cref{eqn:multidse system} and our solution to it in \cref{thm:multidse system solution}.

The work to give this solution will be principally Hopf algebraic, working on the Connes--Kreimer Hopf algebra of rooted trees and decorated generalizations thereof.  In particular, the tubing expansion of \cref{eqn:multidse system} will be constructed in two steps, first understanding and solving similar but simpler equations purely at the level of trees (these are often called combinatorial Dyson--Schwinger equations) and explicitly understanding the form of the map from the Connes--Kreimer Hopf algebra  to the polynomial Hopf algebra which comes from the universal property of the Connes--Kreimer Hopf algebra.

\medskip

The paper proceeds as follows.  We will begin by giving background, defining the Connes-Kreimer Hopf algebra and the combinatorial analogues of our Dyson--Schwinger equations in \cref{subsec CK}, then defining the Faà di Bruno Hopf algebra and looking in more detail at 1-cocycles in \cref{subsec hopf poly}.  In \cref{subsec rio} we look at the close relationship between the renormalization group equation and the Riordan group, a relationship which has been underrecognized up to this point.  We finally define the Dyson--Schwinger equations in \cref{subsec dse} and give a clean, conceptual way of understanding how the invariant charge relates to the renormalization group equation via our work on the Riordan group.  The introduction is then rounded out by \cref{sec:equations for gamma} where we consider the anomalous dimension and \cref{subsec tubings} where we define the tubings that we use for our solutions to Dyson--Schwinger equations and state the previous results in that direction from \cite{tubings}.

We then move to the new work on multiple insertion places, giving the physical context and mathematical set up in \cref{subsec context}, extending our discussion of 1-cocycles to tensor powers of a bialgebra in \cref{subsec tensor 1-cocycles}, generalizing the results regarding the invariant charge and the renormalization group equation in \cref{subsec rge multiple} including proving a conjecture of Nabergall \cite{nabergall:phd}.  In \cref{subsec qle} we consider situations where Dyson--Schwinger equations with multiple insertion places can be transformed into single insertion place Dyson--Schwinger equations, first disproving a different conjecture of Nabergall \cite{nabergall:phd} and then focusing on the special case in which the Dyson--Schwinger equations are in an appropriate sense linear.  Our discussion of multiple insertion places culminates in \cref{subsec tubing multiple} where we show how to give combinatorial series solutions to multiple insertion places Dyson--Schwinger equations via tubings, hence giving combinatorial solutions to all single scale Dyson--Schwinger equations.
We conclude in \cref{sec conclusion}.

\section{Background}

\subsection{The Connes-Kreimer Hopf algebra of rooted trees}\label{subsec CK}
The central algebraic structure for us is the Connes--Kreimer Hopf algebra and its variants.

It will be convenient to think of rooted trees first of all as posets, so define a \textit{rooted forest} to be a finite poset such that each element is covered by at most one element and define a connected forest to be a \textit{rooted tree}. Since we only work with rooted trees and forests we will simply say \textit{tree} and \textit{forest} whenever convenient. 

A rooted tree can be decomposed as a unique maximal
element, the \textit{root}, together with a forest. For a rooted tree $t$ we denote by root by $\rt t$. On the other hand a forest uniquely decomposes as a disjoint union of trees.  A disjoint union of forests is a forest and any upset or downset in a forest is a forest.

When convenient we will also use graph theoretic or tree-specific terminology for forests and trees.  In particular, we often refer to elements of trees as vertices and covering
relations as edges. The unique vertex covering a non-root vertex is its \textit{parent}, vertices covered
by a vertex are its \textit{children}. We will think of a tree as oriented downwards (opposite the
order) so that the number of children of a vertex is the outdegree and is denoted $\od(v)$.

We define the \textit{(undecorated) Connes--Kreimer Hopf algebra}, which we denote $\CK$, as follows.  As an algebra $\CK$ is the vector space freely generated by isomorphism classes of forests with multiplication given by disjoint union.  (Equivalently, $\CK$ is the commutative algebra freely generated by isomorphism classes of (nonempty) trees).  We make $\CK$ into a bialgebra by the coproduct
\[
\Delta(F) = \sum_{D \text{ downset of } F}D \otimes (F\setminus D)
\]
for a forest $F$ and extended linearly.  (This is equivalent to other ways of presenting this coproduct using admissible cuts \cite{connes-kreimer} or antichains of vertices \cite{karenbook} which have appeared in the literature.)  We grade $\CK$ by the number of vertices making $\CK$ into a graded connected\footnote{A graded vector space is connected if the $0$th graded piece is isomorphic to the underlying field. The reader is cautioned not to confuse this with the combinatorial notion of connectedness for trees.} bialgebra and hence a Hopf algebra.

Note that following the same construction with all posets in place of forests, we obtain the standard downset/upset Hopf algebra of posets, $\Posets$.

We can also characterize $\CK$ in a more algebraic way. Consider the linear operator $B_+$ on $\Posets$ which sends each poset $P$ to the poset obtained by adjoining a new element larger than all elements of $P$. Then $\CK$ is the unique minimal subalgebra of $\Posets$ which is mapped to itself by $B_+$. From this perspective it is not immediately obvious that $\CK$ should be a Hopf subalgebra. We can understand this by considering the relationship between $B_+$ and the coproduct. Note that the only downset of $B_+P$ which contains the new element is the
entirety of $B_+P$ . The other downsets coincide with the downsets of $P$, and if $D$ is such a downset we have $(B_+P) \setminus D \cong B_+(P \setminus D)$. It follows that
\[
\Delta B_+P = B_+P \otimes 1 + \sum_{D \text{ downset of } P}D \otimes B_+(P\setminus D)
\]
or in other words
\begin{equation}\label{eq cocycle}
  \Delta B_+ = B_+ \otimes 1 + (\id \otimes B_+)\Delta.
\end{equation}
An operator satisfying \cref{eq cocycle} is a 1-cocycle in the cohomology theory that we will discuss in more detail in \cref{subsec hopf poly}. 

It will be useful to work in a more general setting.  Given a set $I$, we define the
\textit{decorated Connes--Kreimer Hopf algebra} $\CK_I$ as follows. By an
\textit{$I$-tree} (resp.  \textit{$I$-forest}) we mean a tree (resp. forest) with each vertex
decorated by an element of $I$.  We will write $\Trees(I)$ for the set of $I$-trees and
$\Forests(I)$ for the set of $I$-forests.  Then $\CK_I$ is the free vector space on $\Forests(I)$,
made into a bialgebra with disjoint union as multiplication and the same coproduct as in $\CK$ but
preserving the decorations on all vertices. As usual, we can grade by the number of vertices and
find that $\CK_I$ is a connected graded bialgebra and hence a Hopf algebra.

\begin{remark}
    We could instead choose some weight function $w\colon I \to \Nat$ and grade $\CK_I$ by total
    weight. If $w$ takes only \textit{positive} values then this grading will also make $\CK_I$
    connected. In the application to Dyson--Schwinger equations we will have such a weight function
    already given by the loop order of the primitive Feynman diagram corresponding to the 1-cocycle and so it is natural to grade the algebra this way, but it won't really matter for
    anything we do.
\end{remark}

For each $i \in I$, we have an operator $B_+^{(i)}$ on $\CK_I$ that sends an $I$-forest to the
$I$-tree obtained by adding a root with decoration $i$. For the same combinatorial reasons as the
usual $B_+$ operator on $\CK$, all of these are 1-cocycles.

The key significance of the Connes--Kreimer Hopf
algebra is that it possesses a universal property with respect to 1-cocycles (\cite{connes-kreimer}
Theorem 2).  The decorated Connes--Kreimer Hopf algebra likewise has a universal property as follows.

\begin{theorem}[{\cite[Proposition 2 and 3]{foissy20}}] \label{thm:ck universal decorated}
    Let $A$ be a commutative algebra and $\{\Lambda_i\}_{i \in I}$ be a family of linear operators
    on $A$. There exists a unique algebra morphism $\phi\colon \CK_I \to A $ such that $\phi B_+^{(i)}
    = \Lambda_i \phi$. Moreover, if $A$ is a bialgebra and $\Lambda_i$ is a 1-cocycle for each $i$
    then $\phi$ is a bialgebra morphism.
  \end{theorem}

 Note that there is nothing mysterious about the map $\phi$ guaranteed by \cref{thm:ck universal decorated}.
 Since any decorated tree $t$ can be written uniquely (up to reordering) in the form $t = B_+^{(i)}(t_1 \cdots t_k)$ for some $t_1, \ldots  , t_k$ and with $i$ the decoration of the root we can and must define $\phi$ recursively by
\begin{equation}\label{eqn:cocycle morphism recursive}
   \phi(t) = \Lambda_i(\phi(t_1) \cdots \phi(t_k)).
\end{equation}
A natural question is whether we can find an explicit, non-recursive formula for $\phi$. Without
knowing anything about $A$ and the $\Lambda_i$ there is clearly nothing we can do, but the central
insight of \cite{tubings} and the present paper is that in the cases that come up in
Dyson--Schwinger equations we can solve this problem in a nice, combinatorial way.

Already in this context we can build classes of trees using equations involving $B_+$.  These equations have the same recursive structure as the Dyson--Schwinger equations we will ultimately be working on but live strictly in the world of decorated Connes--Kreimer.  They are often called \textit{combinatorial Dyson--Schwinger equations} \cite{yeats:phd}.

Let $P$ be a set
(finite or infinite) and assign each $p \in P$ a weight $w_p \in \Natp$, such that there are only
finitely many elements of each weight, and an \textit{insertion exponent} $\mu_p \in \KK$.\footnote{In the
physical application $P$ is the set of primitive Feynman diagrams, $w_p$ is the dimension of the
cycle space of $p$, and $\mu_p$ is related to the number of ways to insert $P$ into itself, but note that we are not even requiring $\mu_p$
to be an integer. In our approach we essentially treat the insertion exponents as
indeterminates.}
Then we have the single equation combinatorial Dyson--Schwinger equation
\begin{equation} \label{eq:cdse single}
    T(x) = 1 + \sum_{p \in P} x^{w_p} B_+^{(p)}(T(x)^{\mu_p}).
\end{equation}
The solution to this equation is essentially due to Bergbauer and Kreimer \cite{bergbauer-kreimer}
although our statement is somewhat more general. To state it we need some more notation. For a
vertex $v$ with decoration $p$, we will write $w(v) = w_p$ and $\mu(v) = \mu_p$. We will write
\[
    w(t) = \sum_{v \in t} w(v).
\]
Finally, by an \textit{automorphism} of a decorated tree we mean an automorphism of the underlying
tree (as a poset) which preserves the decorations, and as one would expect we denote the
automorphism group of $t$ by $\Aut(t)$.

\begin{proposition}[{\cite[Lemma 4]{bergbauer-kreimer}}] \label{thm:cdse single solution}
    The unique solution to \cref{eq:cdse single} is
    \begin{equation} \label{eq:cdse single soln}
        T(x) = 1 + \sum_{t \in \Trees(P)} \left(\prod_{v \in t} \mu(v)^{\underline{\od(v)}} \right)
        \frac{tx^{w(t)}}{|\Aut(t)|}.
    \end{equation}
\end{proposition}
Here we use the underline notation for falling factorial powers
\[
  a^{\underline{k}} = \prod_{j=0}^{k-1}(a-j).
\]

\begin{proof}
    Let $T(x)$ be given by \cref{eq:cdse single soln}; we will show that $T(x)$ satisfies
    \cref{eq:cdse single}.

    We introduce some notation. For a forest $f \in \Forests(P)$ write $\kappa(f)$ for the number of
    connected components. Write $\Forests_k(P)$ for the set of forests $f \in \Forests(P)$ with
    $\kappa(f) = k$. Write $\tilde T(x) = T(x) - 1$, so $\tilde T(x)$ is a kind of exponential
    generating function for $P$-trees. By the compositional formula (see e.g. \cite[Theorem
    5.5.4]{ec2}), divided powers count forests:
    \[
        \frac{\tilde T(x)^k}{k!} = \sum_{f \in \Forests_k(P)} \left(\prod_{v \in f}
        \mu(v)^{\underline{\od(v)}} \right) \frac{fx^{w(f)}}{|\Aut(f)|}.
    \]
    Then by the binomial series expansion, for any $u \in \KK$ we have
    \begin{align*}
        T(x)^u
        &= (1 + \tilde T(x))^u \\
        &= \sum_{k \ge 0} \binom{u}{k} \tilde T(x)^k \\
        &= \sum_{f \in \Forests(P)} u^{\underline{\kappa(f)}} \left(\prod_{v \in f}
        \mu(v)^{\underline{\od(v)}} \right) \frac{fx^{w(f)}}{|\Aut(f)|}.
    \end{align*}
    Now, any tree $t \in \Trees(P)$ can be uniquely written as $t = B_+^{(p)}f$ for some $p \in P$
    and $f \in \Forests(P)$. In this case, we have $w(t) = w(f) + w_p$. We have $\Aut(t) \cong
    \Aut(f)$ and the outdegrees of all non-root vertices are the same in $t$ as in $f$. The
    outdegree of the root is $\kappa(f)$. Using this bijection we get
    \begin{align*}
        T(x)
        &= 1 + \sum_{p \in P} \sum_{f \in \Forests(P)} \mu_p^{\underline{\kappa(f)}} \left(\prod_{v
        \in f} \mu(v)^{\underline{\od(v)}} \right) \frac{(B_+^{(p)}f)x^{w(f) + w_p}}{|\Aut(f)|}
        \\
        &= 1 + \sum_{p \in P} x^{w_p} B_+^{(p)}(T(x)^{\mu_p})
    \end{align*}
    as desired.
\end{proof}

\begin{example} \label{example:positive s}
    Suppose we have just a single cocycle (so we are essentially in the undecorated Connes--Kreimer
    Hopf algebra $\CK$) and a nonnegative integer insertion exponent $k$, so \cref{eq:cdse single
    soln}
    becomes
    \[
        T(x) = 1 + xB_+(T(x)^k).
    \]
    If we ignore the $B_+$, this would simply give the ordinary generating function for $k$-ary
    trees (in the computer scientists' sense, where the children of each vertex are totally ordered
    including the ``missing'' ones). With the $B_+$ included, it is therefore still a generating
    function for $k$-ary trees but one in which the contribution of each tree is the underlying tree
    itself as an element of $\CK$. It is not too hard to show that $\left(\prod_{v \in t}
    k^{\underline{\od(v)}}\right)/|\Aut(t)|$ counts the number of ways to make a tree $t$ into a
    $k$-ary tree, so this agrees with \cref{eq:cdse single soln}.

    Of particular note is the case $k = 1$, the \textit{linear} Dyson--Schwinger equation. This
    produces \textit{unary} trees, which in the context of quantum field theory are usually referred
    to as \textit{ladders}. (Thinking of trees as posets these are chains; thinking of them as
    graphs they are paths with the root as one of the endpoints.)
\end{example}

\begin{example} \label{example:s = -1}
    Again consider a single cocycle, but now with insertion exponent $-1$. The equation
    \[
        T(x) = 1 + xB_+(T(x)^{-1})
    \]
    can be rewritten in terms of $\tilde T(x) = T(x) - 1$ as
    \[
        \tilde T(x) = xB_+\left(\frac{1}{1 + \tilde T(x)}\right).
    \]
    If the plus sign were replaced by a minus this would give a generating function for plane trees.
    With the plus, it gives plane trees with a sign corresponding to the number of edges. On the
    other hand, noting that $(-1)^{\underline d} = (-1)^d d!$, we have that the contribution of $t$
    to the right side of \cref{eq:cdse single soln} is $(-1)^{|t|-1} \left(\prod_{v \in t}
    \od(v)!\right)/|\Aut(t)|$ which up to the sign is the number of ways to make $t$ into a plane
    tree, so this also matches the formula.
\end{example}

A similar story holds for systems of combinatorial Dyson--Schwinger equations. The setup here is the same, but we partition our indexing set $P$ into $\{P_i\}_{i \in I}$
for some finite set $I$ which will index the equations in the system. Each $p \in P$ is still
assigned a simple weight $w_p \in \Natp$ but the insertion exponent is now an insertion exponent
vector $\mu_p \in \KK^I$.  The system of combinatorial Dyson--Schwinger equations is
\begin{equation} \label{eq:cdse system}
    T_i(x) = 1 + \sum_{p \in P_i} x^{w_p} B_+^{(p)}(\mathbf T(x)^{\mu_p}).
\end{equation}
\label{symboldef:odv}As with the single-equation case, we first need to solve this combinatorial
system.
This solution first appears in \cite{tubings}, though it is a straightforward generalization of \cref{thm:cdse single solution}.
Again, \cite{tubings} only considers a special case where the insertion exponents are constrained but the same proof works in the more general
setup and in fact allowing the insertion exponents to be arbitrary allows the formula
to be much cleaner. We need some additional notation: let us write $\Trees_i(P)$ for the subset of
$\Trees(P)$ for which the root has a decoration in $P_i$; clearly these are the trees that can
contribute to $T_i(x)$. For a vertex $v$ let $\od_i(v)$ be the number of children which have their
decoration in $P_i$; we collect all of these together to form the \textit{outdegree vector} $\odv(v)
\in \Nat^I$.

\begin{theorem}[{\cite[Theorem 4.7]{tubings}}] \label{thm:cdse system solution}
    The unique solution to the system \cref{eq:cdse system} is
    \[
        T_i(x) = 1 + \sum_{t \in \Trees_i(P)} \left(\prod_{v \in t} \mu(v)^{\underline{\odv(v)}}
        \right) \frac{tx^{w(t)}}{|\Aut(t)|}.
    \]
\end{theorem}

\begin{proof}
    Analogous to \cref{thm:cdse single solution}.
\end{proof}

\subsection{Hopf algebras of polynomials and power series}\label{subsec hopf poly}

The next step to bring us closer to the Dyson--Schwinger equations of physics is to understand what will be the target space for our Feynman rules, where our Green functions will live, and relatedly, to what $A$ we will apply the universal property \cref{thm:ck universal decorated}.

We make the polynomial algebra $\KK[L]$ into a (graded) bialgebra by taking
the coproduct and counit as the unique algebra morphisms extending $\Delta L = 1 \otimes L + L \otimes 1$ and $\epsilon(L) = 0$.  This too is a Hopf algebra,
with antipode $S(f (L)) = f (-L)$.
It is often profitable to think of this coproduct in a different way. We can identify
$\KK[L] \otimes \KK[L]$ with $\KK[L_1, L_2]$ (where $L^j \otimes L^k$ corresponds to $L_1^jL_2^k$). The coproduct then simply corresponds to the map $f (L) \mapsto f (L_1 + L_2)$.

We will need some language and standard results on infinitesimal characters. An infinitesimal
character of a bialgebra $H$ is a map $\sigma : H \to \KK$ satisfying
\[
  \sigma(ab) = \sigma(a)\epsilon(b) + \epsilon(a)\sigma(b),
\]
or in other words a derivation of $H$ into the trivial $H$-module $\KK$.  We summarize some useful results for us on infinitesimal characters in the following theorem.
\begin{theorem}\label{thm inf char}
Let $H$ be a graded connected Hopf algebra and $\phi : H \rightarrow \KK[L]$ be an algebra
morphism. Then $\lin \phi$ is an infinitesimal character if and only if $\phi|_{L=0} = \epsilon$, where $\lin$ is the map extracting the linear coefficient. Moreover, the
following are equivalent:
\begin{enumerate}[label=(\roman*),ref=\thetheorem(\roman*)]
\item $\phi$ is a bialgebra morphism,
\item $\phi = \exp_*(L \lin \phi)$,
\item $\phi|_{L=0} = \epsilon$ and $\frac{d}{dL} \phi = (\lin \phi) * \phi$,
\item $\phi|_{L=0} = \epsilon$ and $\frac{d}{dL} \phi = \phi * (\lin \phi)$.
\end{enumerate}
\end{theorem}
The proof of this theorem consists of standard calculations and can also be found in \cite[Section
2.2.4]{OH:phd}.

Another important Hopf algebra is the Faà di Bruno Hopf algebra.

\label{symboldef:fdb}Let $\widetilde \DiffGp \subset \KK[[x]]$ be the set of all series with zero
constant term and nonzero linear term. These are also known as \textit{formal diffeomorphisms} and
form a group with respect to composition of series. Let $\DiffGp$ be the subgroup consisting of
series with linear term $x$; these are sometimes known as \textit{$\delta$-series}. It turns out
that $\DiffGp$ is essentially isomorphic to the character group of a graded Hopf algebra, the
\textit{Faà di Bruno Hopf algebra} $\FdB$.

As an algebra, $\FdB$ should be thought of as the algebra of polynomial functions on
$\DiffGp$. Explicitly, it is the polynomial algebra $\KK[\pi_1, \pi_2, \dots]$ in an
$\Natp$-indexed set of variables. We organize these variables into a power series
\[
    \Pi(x) = x + \sum_{n \ge 1} \pi_nx^{n+1}.
\]
Then the map from the character group $\Ch(\FdB) \to \DiffGp$ given by $\zeta \mapsto \zeta(\Pi(x))$ is clearly a
bijection. (Note that here and throughout notation like $\zeta(\Pi(x))$ implicitly means
applying $\zeta$ coefficientwise.) We define a coproduct
\begin{equation} \label{eqn:fdb coproduct}
    \Delta \pi_n = \sum_{k=0}^n [x^{n+1}] \Pi(x)^{k+1} \otimes \pi_k.
\end{equation}
(where $\pi_0 = 1$). Observe that this makes $\FdB$ into a connected graded bialgebra if we define
$\pi_n$ to have degree $n$; this is the reason for the off-by-one in the definition. The following
proposition is essentially immediate from \cref{eqn:fdb coproduct}.

\begin{proposition} \label{thm:fdb convolution}
    Let $A$ be a commutative algebra and $\phi, \psi\colon \FdB \to A$ be algebra morphisms. Let
    $\Phi(x) = \phi(\Pi(x))$ and $\Psi(x) = \psi(\Pi(x))$. Then $(\phi * \psi)(\Pi(x)) = \Psi(\Phi(x))$.
\end{proposition}

In particular, this actually implies that the map $\Ch(\FdB) \to \DiffGp$ described above is
an \textit{anti}-isomorphism of groups. Clearly we could have defined the coproduct with the tensor
factors flipped in order to make it an isomorphism, but the way we have defined it is both
traditional and will turn out to be convenient for our purposes.

Comodules over a coalgebra become modules over the dual. Let $C$ be a coalgebra and $M$ be a left
$C$-comodule with coaction $\delta$. For $\alpha$ an element in the dual ($\alpha \in C^*$), define
$m \rightact \alpha = (\alpha \otimes \id_M )\delta(m)$.  We can compute
\begin{align*}
(m \rightact \alpha) \rightact \beta & = (\beta \otimes \id_M )\delta((\alpha \otimes id_M )\delta(m)) \\
& = (\beta \otimes \id_M )(\alpha \otimes \delta)\delta(m) \\
& = (\alpha \otimes \beta \otimes \id_M )(\id_C \otimes \delta)\delta(m) \\
& = (\alpha \otimes \beta \otimes \id_M )(\Delta_C \otimes \id_M )\delta(m) \\
& = ((\alpha * \beta) \otimes \id_M )\delta(m) \\
  & = m \rightact (\alpha * \beta)
\end{align*}
so $\rightact$ makes $M$ into a right module over $C^*$. Analogously, if $M$ is a right $C$-comodule
then we define $\alpha \leftact m = (\id_M \otimes \alpha)\delta(m)$ and this makes $M$ into a left
$C^*$-module. In particular, $C$ is both a left and right comodule over itself and hence both a left
and right module over $C^*$.

Applying this to $\FdB^*$ and observing that it is sufficient to understand how an element of
$\FdB^*$ acts on the generators, or equivalently on the series $\Pi(x)$ itself, we obtain the following result directly from \cref{eqn:fdb coproduct}.

\begin{proposition} \label{thm:fdb actions}
    Suppose $\phi \in \FdB^*$ and let $\Phi(x) = \phi(\Pi(x))$. Then:
    \begin{enumerate}[label=(\roman*),ref=\thetheorem(\roman*)]
        \item \label[proposition]{thm:fdb action left}
            $\phi \leftact \Pi(x) = \Phi(\Pi(x))$.
        \item
            If $\phi \in \Ch(\FdB)$ then $\Pi(x) \rightact \phi = \Pi(\Phi(x))$.
        \item \label[proposition]{thm:fdb inf char right}
            If $\phi \in \ch(\FdB)$ then $\Pi(x) \rightact \phi = \Phi(x)\Pi'(x)$, where $\ch(\FdB)$ is the Lie algebra of infinitesimal characters.
    \end{enumerate}
\end{proposition}

As a particular consequence of \cref{thm:fdb inf char right}, we get a nice description of the Lie
algebra $\ch(\FdB)$: the map $\phi \mapsto \phi(\Pi(x))\frac{d}{dx}$ gives a faithful representation
by differential operators on $\KK[[x]]$. We can also combine this with \cref{thm inf char} to characterize bialgebra morphisms $\FdB \to \KK[L]$.

\begin{theorem} \label{thm:fdb pde}
    Let $\phi \colon \FdB \to \KK[L]$ be an algebra morphism and let $\Phi(x, L) = \phi(\Pi(x))$.
    Let $\beta(x)$ be the linear term in $L$ of $\Phi(x, L)$. Then $\phi$ is a bialgebra morphism
    if and only if $\Phi(x, 0) = x$ and
    \begin{equation} \label{eqn:fdb pde}
        \frac{\partial \Phi(x, L)}{\partial L} = \beta(x)\frac{\partial \Phi(x, L)}{\partial x}.
    \end{equation}
\end{theorem}

\begin{proof}
  Recall the notation and results of \cref{thm:fdb inf char right}.
    We know that $\phi$ is a bialgebra morphism if and only if $\phi|_{L=0} = \counit$ and $\frac{d}{dL}\phi = (\lin \phi) * \phi$. Since $\phi$ is an algebra morphism, its
    behaviour is determined by what it does to the generators, so these are respectively equivalent
    to $\Phi(x, 0) = \counit(\Pi(x)) = x$ and
    \[
        \frac{\partial \Phi(x, L)}{\partial L} = ((\lin \phi) * \phi)(\Pi(x)) = \phi(\Pi(x)
        \rightact \lin(\phi)).
    \]
    Since $\lin \phi$ is an infinitesimal character, by \cref{thm:fdb inf char right} the right-hand
    side is $\beta(x) \frac{\partial \Phi(x,L)}{\partial x}$ as wanted.
\end{proof}

\begin{example}
  For example, take $\phi\colon \FdB \to \KK[L]$ defined on generators by $\phi(\pi_n) = L^n$ and extended as an algebra morphism.  Then $\Phi(x,L) = \sum_{n=0}^{\infty} L^nx^{n+1} = \frac{x}{1-Lx}$ and $\beta(x) = x^2$.  We can now check the conditions from \cref{thm:fdb pde}: we have $\Phi(x,0) = x$ and $\frac{\partial \Phi(x,L)}{\partial L} = x^2\sum_{n=1}^\infty nL^{n-1}x^{n-1} = \beta(x)\frac{\partial \Phi(x,L)}{\partial x}$, hence $\phi$ is a bialgebra morphism. 
\end{example}

We mentioned in \cref{subsec CK} that the $B_+$ and $B_+^{(i)}$ operators on rooted trees and decorated rooted trees are 1-cocycles.  We will also need to understand 1-cocycles on other bialgebras, so let us now discuss in more detail.

Let $H$ be a bialgebra and $M$ a left
comodule over $H$, with coaction $\delta$. For $k \ge 0$, a \textit{$k$-cochain on $M$} is a linear
map $M \to H^{\otimes k}$. Denote the vector space of $k$-cochains by $\Chains^k(H, M)$. The
\textit{coboundary map} $d_k\colon \Chains^k(H, M) \to \Chains^{k+1}(H, M)$ is defined by
\[
    d_k \Lambda = (\id_H \otimes \Lambda)\delta + \sum_{j=1}^k (-1)^j (\id_H^{\otimes(j-1)}
    \otimes \Delta \otimes \id_H^{\otimes(k-j)}) \Lambda + (-1)^k \Lambda \otimes 1.
\]
\label{symboldef:cocycles}The kernel and image of this map are the spaces of $k$-cocycles and
$(k+1)$-coboundaries respectively. The space of $k$-cocycles is denoted $\Cycles^k(H, M)$. A tedious
but routine calculation shows that $d_{k+1}d_k = 0$, so every coboundary is a cocycle. The quotient
$\Homol^k(H, M) = \Cycles^k(H, M)/d_{k-1}\Chains^{k-1}(H, M)$ is the $k$th cohomology of the
comodule $M$. Most often considered is the case $M = H$, in which case we simply write
$\Cycles^k(H)$ and $\Homol^k(H)$.

Note that $\Homol^k(H, M)$ is $\mathrm{Ext}_H^k(M, H)$ in the category of comodules over $H$.  The original notion of cohomology for coalgebras introduced by Doi \cite{doi} worked with a bicomodule with a left coaction and a right coaction.  Our definition is the special case where the right coaction is trivial.

We will only be interested in the case $k = 1$. Moreover, for the remainder of this section we will
focus on the case $M = H$. (We will consider some other comodules in \cref{subsec tensor 1-cocycles}.) In this case the cocycle condition $d_1\Lambda = 0$ can be written
\[
    \Delta \Lambda = \Lambda \otimes 1 + (\id_H \otimes \Lambda)\Delta
\]
which is the form we saw for $B_+$. Here is another very natural example.

\begin{example} \label{example:integral cocycle}
    Let $\Int$ be the integration operator on $\KK[L]$:
    \[
        \Int f(L) = \int_0^L f(u)\,du.
    \]
    Recall that the coproduct on $\KK[L]$ can be interpreted as
    substituting a sum $L_1 + L_2$ in place of the variable $L$. That $\Int$ is a 1-cocycle then
    boils down to some familiar properties of integrals:
    \begin{align*}
        \int_0^{L_1 + L_2} f(u)\,du
        &= \int_0^{L_1} f(u)\,du + \int_{L_1}^{L_1 + L_2} f(u)\,du \\
        &= \int_0^{L_1} f(u)\,du + \int_0^{L_2} f(L_1 + u)\,du.
    \end{align*}

    By the universal property, $\Int$ defines a morphism $\phi\colon \CK \to \KK[L]$. An easy
    induction with the recurrence \cref{eqn:cocycle morphism recursive} gives
    \[
        \phi(t) = \frac{L^{|t|}}{\prod_{v \in t} |t_v|}
    \]
    where $t_v$ denotes the subtree (principal downset) rooted at $v$. The denominator is also known
    as the \textit{tree factorial}. A formula of Knuth \cite[Section 5.1, Exercise 20]{aocp3} gives
    the number of linear extensions of a tree as
    \[
        \linex(t) = \frac{|t|!}{\prod_{v \in t} |t_v|}
    \]
    and hence we can alternatively write
    \[
        \phi(t) = \frac{\linex(t)L^{|t|}}{|t|!}.
    \]
    This latter formula is the simplest special case of the formula we will derive for arbitrary
    1-cocycles on $\KK[L]$ in the context of the tubing expansion.
\end{example}

We state some basic properties of 1-cocycles. For stating these it is convenient to generalize
the notion of convolution to maps defined on comodules: for an algebra $A$ and maps $\alpha\colon H
\to A$ and $\beta\colon M \to A$, write
\[
    \alpha *_\delta \beta = m_A(\alpha \otimes \beta)\delta
\]
where $m_A$ is the multiplication map on $A$.

\begin{lemma} \label{thm:basic cocycle properties}
    Let $M$ be a comodule over $H$ and $\Lambda \in \Cycles^1(H, M)$. Then:
    \begin{enumerate}[label=(\roman*), ref=\thetheorem(\roman*)]
        \item
            \label[lemma]{thm:cocycle convolution} If $\alpha, \beta\colon H \to A$ for some algebra
            $A$ then $(\alpha * \beta)\Lambda = \beta(1) \alpha \Lambda + \alpha *_\delta
            \beta\Lambda$.
        \item
            \label[lemma]{thm:cocycle counit} $\counit \Lambda = 0$.
        \item
            \label[lemma]{thm:cocycle contravariant} If $\phi\colon N \to M$ is a homomorphism of
            comodules then $\Lambda\phi \in \Cycles^1(H, N)$.
    \end{enumerate}
\end{lemma}

\begin{proof}
    \begin{enumerate}[label=(\roman*)]
        \item
            Immediate from the definition of 1-cocycles and convolution.
        \item
            Follows from (i) since $\counit$ is the identity for convolution.
        \item
            Write $\delta$ and $\delta'$ for the coactions on $M$ and $N$ respectively. We have
            \[
                \Delta \Lambda \phi = (\Lambda \otimes 1)\phi + (\id \otimes \Lambda)\delta \phi =
                \Lambda\phi \otimes 1 + (\id \otimes \Lambda \phi) \delta'. \qedhere
            \]
    \end{enumerate}
\end{proof}

Note in particular that if $\beta(1) = 0$ (e.g. if $\beta$ is an infinitesimal character) then (i)
just says $(\alpha * \beta)\Lambda = \alpha * \beta\Lambda$.

Suppose $\Lambda \in \Cycles^1(H)$. We can use $\Lambda$ to build new cocycles on various comodules.
Given a left comodule $M$ with coaction $\delta$ and a linear map $\psi\colon M \to \KK$, we define
$\Lambda \actoncocycle \psi = (\Lambda \otimes \psi)\delta$.

\begin{lemma} \label{thm:actoncocycle}
    Subject to the above assumptions, $\Lambda \actoncocycle \psi \in \Cycles^1(H, M)$.
\end{lemma}

\begin{proof}
    We compute
    \begin{align*}
        \Delta (\Lambda \actoncocycle \psi)
        &= (\Delta\Lambda \otimes \psi) \delta \\
        &= (\Lambda \otimes \psi)\delta \otimes 1 + ((\id \otimes \Lambda)\Delta \otimes \psi)
        \delta \\
        &= (\Lambda \actoncocycle \psi) \otimes 1 + (\id \otimes \Lambda \otimes \psi)(\Delta
        \otimes \id)\delta \\
        &= (\Lambda \actoncocycle \psi) \otimes 1 + (\id \otimes \Lambda \otimes \psi)(\id
        \otimes \delta)\delta \\
        &= (\Lambda \actoncocycle \psi) \otimes 1 + (\id \otimes (\Lambda \actoncocycle
        \psi))\delta. \qedhere
    \end{align*}
\end{proof}

As a special case of this, note that $d\counit \actoncocycle \psi = d\psi$ (dropping the subscript to the coboundary map by the usual abuse of notation). When $M = H$ we can
write $\Lambda \actoncocycle \psi$ using the left action of $H^*$ on $H$ described above:
\[
    (\Lambda \actoncocycle \psi)h = \Lambda(\psi \leftact h).
\]
Using this operation and the integral cocycle from \cref{example:integral cocycle}, we can describe
all 1-cocycles on $\KK[L]$.

\begin{theorem}[Panzer {\cite[Theorem 2.6.4]{panzer-masters}}] \label{thm:polynomial cocycles}
    For any series $A(L) \in \KK[[L]]$, the operator
    \begin{equation} \label{eqn:1-cocycle integral}
        f(L) \mapsto \int_0^L A(d/du) f(u) \,du
    \end{equation}
    is a 1-cocycle on $\KK[L]$. Moreover, all 1-cocycles on $\KK[L]$ are of this form.
\end{theorem}

\begin{corollary}
    The cohomology $\Homol^1(\KK[L])$ is 1-dimensional and generated by the class of the integral
    cocycle $\Int$.
\end{corollary}

\begin{proof}
    Note $\Int(1) = L$, so $\Int$ is not a coboundary. Now suppose $\Lambda$ is a 1-cocycle given by
    \cref{eqn:1-cocycle integral}. Write $A(L) = a_0 + LB(L)$ for some series $B(L)$. Then we have
    \begin{align*}
        \Lambda f(L)
        &= \int_0^L A(d/du) f(u)\,du \\
        &= a_0 \int_0^L f(u)\,du + \int_0^L \frac{d}{du}B(d/du)f(u)\,du \\
        &= a_0 \Int f(L) + B(d/dL)f(L) - B(d/dL)f(L)\big|_{L=0}
    \end{align*}
    hence $\Lambda = a_0\Int + d\beta$ where $\beta$ is the linear form $L^n \mapsto [L^n]B(L)$.
\end{proof}

\begin{remark} \label{remark:differential cocycle form}
    We can also write 1-cocycles on $\KK[L]$ in a different form, namely
    \[
        f(L) \mapsto f(\partial/\partial \rho) \frac{e^{L\rho} - 1}{\rho}A(\rho)\big|_{\rho = 0}.
    \]
    Checking on the basis of monomials quickly reveals that this is equivalent to the 1-cocycle that
    appears in the statement of \cref{thm:polynomial cocycles}. Operators of this form are often used by one of us in formulating Dyson--Schwinger equations, see for instance \cite{yeats:phd, marie-yeats, hihn-yeats}.
\end{remark}

In view of the comments above, it will turn out that the key to solving Dyson--Schwinger equations combinatorially will be in determining an explicit formula for the map $\CK \to \KK[L]$ induced by
$\Lambda$, in terms of the coefficients of the series $A(L)$.

This set up generalizes immediately to the case of the decorated Connes--Kreimer Hopf algebra $\CK_I$ and the 1-cocycles $B_+^{(i)}$ for $i\in I$.

\subsection{The renormalization group equation and the Riordan group}\label{subsec rio}

Let $\beta(x)$ and $\gamma(x)$ be formal power series, with $\beta(0) = 0$. The
\textit{renormalization group equation} (RGE) (or \textit{Callan--Symanzik equation}) is
\begin{equation} \label{eqn:rge}
    \left(\frac{\partial}{\partial L} - \beta(x) \frac{\partial}{\partial x} - \gamma(x)\right)G(x,
    L) = 0.
\end{equation}
As suggested by the notation, we will ultimately want to think of this $G(x, L)$ as the same one
which appears in the Dyson--Schwinger equation, but for the purposes of this section we can
consider it to be simply notation for the (potential) solution to this PDE.

\begin{remark}\label{rmk rge meaning}
  In the quantum field theory context, the renormalization group equation \cref{eqn:rge} is a very important differential equation because it describes how the Green function changes as the energy scale $L$ changes.

  The series $\beta(x)$ is the \emph{beta function} of the quantum field theory.  Thinking for a moment not in terms of formal power series, but in the physical context with functions, then $\beta$ should be the $L$ derivative of the coupling $x$.  Returning to the series context, this becomes essentially the linear term of what we will soon call the invariant charge, $Q$, see \cref{eqn:definition of q}.  The beta function is important physically because it describes how the coupling (which determines the strength of interactions) changes with the energy.  Zeros of the beta function are particularly important since the situation simplifies at such point.

  The series $\gamma(x)$ is the \emph{anomalous dimension}. It is a dimension in the sense of scaling dimension and it is anomalous in the sense that in the classical setting one would have a constant integer in place of $\gamma(x)$.
  A particularly easy case is when $\beta(x)=0$ in  \cref{eqn:rge} as there the differential equation can be solved by $e^{\gamma(x)L}$ which is of a particularly simple form sometimes called a scaling solution.

  Returning to our formal series context, taking the coefficient of $L^0$ in \cref{eqn:rge} we see that the anomalous dimension is nothing other than the linear term in $L$ of $G$ (in the cases of physical interest we will always have $g_0=1$).  Furthermore, writing $G(x,L) = 1+\sum_{i\geq 0}g_i(x)L^i$ (so $\gamma(x) = g_1(x)$) and taking the coefficient of $L^{k-1}$ in \cref{eqn:rge} we obtain
  \[
  g_k(x) = \frac{1}{k}\left(\beta(x)\frac{d}{dx} + \gamma(x)\right)g_{k-1}(x)
  \]
  so we see that knowing $\gamma(x)$ and $\beta(x)$ is sufficient to recursively determine all the $g_k(x)$ and hence to determine $G(x,L)$.  Furthermore, since $\beta(x)$ is essentially the linear term of the invariant charge, in the single equation case $\beta(x)$ is just a normalization away from anomalous dimension $\gamma(x)$ and in the system case is a linear combination of the anomalous dimensions of the different $G_i(x,L)$ in the system (see \cite{yeats:phd}).  Overall, then we conclude that knowing the anomalous dimension(s) suffices to determine the Green functions.  Usually we will none the less work at the level of the Green function, but sometimes it will be convenient to work only with the anomalous dimension, which we are free to do since this does not lose any information.
\end{remark}

The goal of this
section is to explain how \cref{eqn:rge} is intimately related to a certain Hopf algebra. As a
starting point, notice that if $\gamma(x) = 0$ we have already seen this equation: by \cref{thm:fdb
pde} it describes a bialgebra morphism $\FdB \to \KK[L]$. We will show that something similar is
true for \cref{eqn:rge}. 

Recall that $\widetilde \DiffGp$ denotes the group of formal power series with
zero constant term and nonzero linear term under composition and $\DiffGp$ the subgroup of
$\delta$-series, and that $\DiffGp^\mathrm{op}$ is isomorphic to the character group of $\FdB$. Now
observe that for $\Phi(x) \in \widetilde\DiffGp$ the map $F(x) \mapsto F(\Phi(x))$ is a ring
automorphism of $\KK[[x]]$. Moreover, composing these automorphisms corresponds to composing the
series in reverse, so $\widetilde \DiffGp^\mathrm{op}$ (and hence also $\DiffGp^\mathrm{op}$) acts
by automorphisms on $\KK[[x]]$. Consequently they also act on $\KK[[x]]^\times$, the multiplicative
group of power series with nonzero constant term. Let $\KK[[x]]^\times_1$ be the subgroup of
$\KK[[x]]^\times$ consisting of those series with constant term 1. The \textit{Riordan group} is the
semidirect product $\RioGp = \KK[[x]]^\times_1 \rtimes \DiffGp$. Explicitly, the elements consist of
pairs $(F(x), \Phi(x))$ of series with $F(x) \in \KK[[x]]^\times_1$ and $\Phi(x) \in \DiffGp$, with
the operation
\[
    (F(x), \Phi(x)) * (G(x), \Psi(x)) = (F(x)G(\Phi(x)), \Psi(\Phi(x))).
\]

\begin{remark}
    The Riordan group was first introduced---at least under that name---by Shapiro, Getu, Woan, and
    Woodson \cite{sgww}. It is usually thought of as a group of infinite matrices, via the
    correspondence
    \[
        (F(x), \Phi(x)) \mapsto \Big[ [x^i] F(x)G(x)^j \Big]_{i,j \in \Nat}
    \]
    sending a pair of series to their \textit{Riordan matrix}. (This is simply a matrix
    representation of the natural action of $\RioGp$ on $\KK[[x]]$.) Conventions vary on whether or
    not to include the restrictions on coefficients; our choice matches the original definition in
    \cite{sgww} as well as being convenient for relating $\RioGp$ to a Hopf algebra.
\end{remark}

\label{symboldef:rio}We now wish to define a Hopf algebra with $\RioGp$ as its character group,
similar to the Faà di Bruno Hopf algebra. We will call it the \textit{Riordan Hopf algebra} and
denote it by $\RioHopf$.  As an algebra, $\RioHopf$ is a free commutative algebra in two sets of
generators $\{\pi_1, \pi_2, \dots\}$ and $\{y_1, y_2, \dots\}$. The $\pi$'s will generate a copy of
$\FdB$; in particular, their coproduct is still given by \cref{eqn:fdb coproduct}. (This inclusion
$\FdB \to \RioHopf$ is dual to the quotient map $\RioGp \to \DiffGp^\mathrm{op}$ coming from the
semidirect product.) We assemble the $y$'s into a power series as well, this time in the more
obvious way:
\[
    Y(x) = 1 + \sum_{n \ge 1} y_n x^n.
\]
Then the coproduct is given by
\begin{equation} \label{eqn:rio coproduct}
    \Delta y_n = \sum_{j = 0}^n [x^n] Y(x)\Pi(x)^j \otimes y_j.
\end{equation}

Analogously to \cref{thm:fdb convolution}, we easily get the following result.

\begin{proposition} \label{thm:rio convolution}
    Let $A$ be a commutative algebra and $\phi, \psi\colon \RioHopf \to A$ be algebra morphisms. Let
    $F(x) = \phi(Y(x))$, $\Phi(x) = \phi(\Pi(x))$, $G(x) = \psi(Y(x))$, and $\Psi(x) =
    \psi(\Pi(x))$. Then
    \[
        (\phi * \psi)(Y(x)) = F(x)G(\Phi(x))
    \]
    and
    \[
        (\phi * \psi)(\Pi(x)) = \Psi(\Phi(x)).
    \]
    Consequently, $\Ch(\RioHopf) \cong \RioGp$.
\end{proposition}

We also have an analogue of \cref{thm:fdb actions}. Note that since the $\pi$'s generate a copy of
$\FdB$ we can simply apply \cref{thm:fdb actions} itself to see how elements of the dual act on
them. Thus we only need the actions on $Y(x)$.

\begin{proposition}
    Suppose $\phi \in \RioHopf^*$ and let $F(x) = \phi(Y(x))$ and $\Phi(x) = \phi(\Pi(x))$. Then
    \begin{enumerate}[label=(\roman*),ref=\thetheorem(\roman*)]
        \item \label[proposition]{thm:rio action left}
            $\phi \leftact Y(x) = F(\Pi(x))Y(x)$.
        \item
            If $\phi \in \Ch(\FdB)$ then $Y(x) \rightact \phi = F(x)Y(\Phi(x))$.
        \item \label[proposition]{thm:rio inf char right}
            If $\phi \in \ch(\FdB)$ then $\Pi(x) \rightact \phi = \Phi(x)Y'(x) +  F(x)Y(x)$.
    \end{enumerate}
\end{proposition}

Finally we reach the main result of this section.

\begin{theorem} \label{thm:rio pde}
    Let $\phi \colon \RioHopf \to \KK[L]$ be an algebra morphism and let $F(x, L) = \phi(Y(x))$ and
    $\Phi(x, L) = \phi(\Pi(x))$. Let $\beta(x)$ be the linear term in $L$ of $\Phi(x, L)$ and
    $\gamma(x)$ the linear term in $L$ of $F(x, L)$. Suppose $\phi$ is a bialgebra morphism when
    restricted to the subalgebra $\FdB$. Then $\phi$ is a bialgebra morphism on all of $\RioHopf$ if
    and only if $F(x, L)$ satisfies the renormalization group equation
    \begin{equation} \label{eqn:rio pde}
        \left(\frac{\partial}{\partial L} - \beta(x) \frac{\partial}{\partial x} -
        \gamma(x)\right)F(x, L) = 0.
    \end{equation}
\end{theorem}

\begin{proof}
    By the same argument as \cref{thm:fdb pde} it is necessary and sufficient to have
    $\frac{d}{dL}\phi = (\lin \phi) * \phi$. Applying \cref{thm:rio inf char right} gives the
    result.
\end{proof}

Obviously, if we assume that $\phi$ is merely an algebra map $\RioHopf \to \KK[L]$ then it is a
bialgebra morphism if and only if it satisfies the conditions of both \cref{thm:fdb pde} and
\cref{thm:rio pde}.

\begin{remark}
    A result equivalent to \cref{thm:rio pde} was proved by Bacher \cite[Proposition 7.1]{bacher06}.
    He does not take a Hopf algebra perspective but instead essentially works with the Lie algebra
    $\ch(\RioHopf)$ in a matrix representation and for an element $\sigma \in \ch(\RioHopf)$
    corresponding to the pair $(\gamma(x), \beta(x))$ characterizes $\exp_*(L\sigma)$ as (the
    Riordan matrix of) the solution to \cref{eqn:rio pde} and \cref{eqn:fdb pde}, which is
    equivalent to our result by \cref{thm inf char}. That the PDE in question is in fact
    the renormalization group equation seems not to have been noticed.
\end{remark}

The key insight here is that working with $\RioHopf$ lets us separate out $Y$ and $\Pi$ in a way that's ideally suited for understanding the renormalization group equation, and which we will use to understand the role of the invariant charge in the following.

\subsection{Dyson--Schwinger equations}\label{subsec dse}

Now we are ready to give the honest Dyson--Schwinger equations, not just their combinatorial avatars, in the form in which we will use them.

As in the combinatorial set up let $P$ be a set
(finite or infinite) and assign each $p \in P$ a weight $w_p \in \Natp$, such that there are only
finitely many elements of each weight, and an insertion exponent $\mu_p \in \KK$.
 To each $p \in P$ we also associate a 1-cocycle $\Lambda_p$ on the polynomial Hopf
algebra $\KK[L]$. The \textit{Dyson--Schwinger equation} (DSE) defined by these data is
\begin{equation} \label{eq:dse lambda single}
    G(x, L) = 1 + \sum_{p \in P} x^{w_p} \Lambda_p(G(x, L)^{\mu_p}).
\end{equation}
(Note that here and throughout, expressions such as $\Lambda_p(G(x, L)^{\mu_p})$ are to
be interpreted as meaning that we expand the argument as a series in $x$ and apply the operator
coefficientwise.)  Using the same data but with $B_+^{(p)}$ in place of $\Lambda_p$ we see that the corresponding combinatorial Dyson--Schwinger equation is the one in \cref{eq:cdse single}.

By \cref{thm:polynomial cocycles} we can write
\begin{equation} \label{eqn:reminder of polynomial cocycles}
    \Lambda_p f(L) = \int_0^L A_p(d/du) f(u)\,du
\end{equation}
for some series $A_p(L) \in \KK[[L]]$ which physically is more or less the Mellin transform of $p$.

A particularly nice case, which covers most of the physically reasonable examples, is when there is
a linear relationship between the weights and the insertion exponents: $\mu_p = 1 + sw_p$ for some
$s \in \KK$. In this case we can combine terms to get
\begin{equation} \label{eq:dse lambda single simple}
    G(x, L) = 1 + \sum_{k \ge 1} x^k \Lambda_k(G(x, L)^{1+sk}).
\end{equation}
Previous work on combinatorial solutions to Dyson--Schwinger equations has focused on this case, and
indeed equations of this form have some special properties which we discuss below. However, the tubing solutions apply in the more general form \cref{eq:dse lambda single}.

\begin{remark} \label{remark:physical dses}
  These Dyson--Schwinger equations are not yet quite in the form one would usually see in the physics literature.  As a first step, using \cref{remark:differential cocycle form}, we obtain the Dyson--Schwinger equations in the form usually presented by one of us \cite{yeats:phd, marie-yeats, hihn-yeats}.  Using \cref{eqn:reminder of polynomial cocycles} brings us closer to the integral equation form that perturbative Dyson--Schwinger equations are typically written in.  See \cite{yeats:phd} or \cite{OH:phd} for a derivation relating them.  The textbook presentation of Dyson--Schwinger equations is often one step further distant. Taking a perturbative or diagrammatic expansion (along with the usual techniques to reduce to one particle irreducible Green functions) bridges this last gap, see for example \cite{Sprimer}.

  The diagrammatic form of the Dyson--Schwinger equations mentioned above is perhaps the easiest
  perspective to get an intuition for what these equations are telling us---they are telling us how
  to build all Feynman diagrams contributing to a given process by inserting simpler
  primitive\footnote{primitive in a renormalization Hopf algebra, or from a physics perspective, subdivergence-free} Feynman diagrams into themselves.
  This explains some otherwise mysterious aspects of the nomenclature.  The indexing set $I$ for
  systems is typically giving the external edges of the diagram.  The weight is the loop order
  (dimension of the cycle space) of the primitive diagram.  The insertion exponent counts how the
  number of places a Feynman diagram of the given external edge structure can be inserted grows as
  the loop order grows.  The function $A(\rho)/\rho$ is the Mellin transform of the Feynman integral
  for the primitive diagram regularized at the insertion place.
\end{remark}

As before, we are interested not only in single equations but also
systems.  In that case we partition our indexing set $P$ into $\{P_i\}_{i \in I}$
for some finite set $I$ which will index the equations in the system. Each $p \in P$ is still
assigned a simple weight $w_p \in \Natp$ but the insertion exponent is now an insertion exponent
vector $\mu_p \in \KK^I$. We are now solving for a vector of series $\mathbf G(x, L) = (G_i(x,
L))_{i \in I}$. The system of equations we consider is
\begin{equation} \label{eq:dse lambda system}
    G_i(x, L) = 1 + \sum_{p \in P_i} x^{w_p} \Lambda_p(\mathbf G(x, L)^{\mu_p}).
\end{equation}
The corresponding combinatorial Dyson--Schwinger equation we saw in \cref{eq:cdse system}.

The analogue of the special case \cref{eq:dse lambda single simple} is the existence of a so-called
\textit{invariant charge} for the system, which we define in the following and which is closely related to $G(x,L)$ satisfying a renormalization group equation.

We will begin with the simplest case \cref{eq:dse lambda single}. By \cref{thm:rio pde}, we see that
$G(x, L)$ satisfies a renormalization group equation if there exists a bialgebra morphism $\RioHopf
\to \KK[[L]]$ that sends $Y(x)$ to $G(x, L)$. It is natural to lift to the combinatorial equation
\cref{eq:cdse single} and ask instead for a bialgebra morphism $\RioHopf \to \CK_P$ that sends
$Y(x)$ to $T(x)$. The question then is where $\Pi(x)$ should be mapped. We wish to construct from
$T(x)$ an auxiliary series $Q(x) \in \CK_P[[x]]$---the invariant charge---such that the map sending
$\Pi(x)$ to $Q(x)$ and $Y(x)$ to $T(x)$ is a bialgebra morphism. Note that this is unique if it
exists since the coproduct formula \cref{eqn:rio coproduct} allows us to recover it from the
coproducts of coefficients of $T(x)$. It turns out that the case in which we can ensure this exists
is exactly the special case \cref{eq:dse lambda single simple}.\footnote{We will only show
sufficiency; for necessity see \cite[Proposition 10]{foissy14} although note that the setup there is
somewhat different from ours.}

\begin{proposition} \label{thm:dse single rio}
    Let $T(x) \in \CK_{\Natp}[[x]]$ be the solution of the combinatorial Dyson--Schwinger equation
    \[
        T(x) = 1 + \sum_{k \ge 1} x^kB_+^{(k)}(T(x)^{1+sk}).
    \]
    Then the algebra morphism $\phi\colon \RioHopf \to \CK_{\Natp}$ defined by $\phi(Y(x)) = T(x)$
    and $\phi(\Pi(x)) = xT(x)^s$ is a bialgebra morphism. As a consequence, the solution $G(x, L)$
    to the corresponding Dyson--Schwinger equation \cref{eq:dse lambda single simple} satisfies the
    renormalization group equation
    \[
        \left(\frac{\partial}{\partial L} - sx\gamma(x)\frac{\partial}{\partial x} -
        \gamma(x)\right)G(x, L) = 0
    \]
    where $\gamma(x)$ is the linear term in $L$ of $G(x, L)$.
\end{proposition}

\begin{remark} \label{rem:nothing new under the delta}
    While the \textit{phrasing} of \cref{thm:dse single rio} seems to be new, its content is not:
    the coproduct formula for $T(x)$ implied by combining this result with \cref{eqn:rio coproduct}
    is well-known. (See \cite[Lemma 4.6]{yeats:phd} for exactly this formula and for instance
    \cite[Theorem 1]{borinsky14}, \cite[Proposition 4.2]{prinz22}, and \cite[Proposition
    7]{vansuij09} for essentially equivalent formulas appearing in slightly different contexts.)
    Our proof is in some sense also the same as what had appeared before, but we believe this presentation is more conceptually clear.  The
    generalization to distinguished insertion places in \cref{subsec rge multiple} is new.
\end{remark}

We now work towards proving \cref{thm:dse single rio}. It will be convenient to abuse notation here by neglecting to notate the obvious (but non-injective!) map from tensor products of power series to power series with tensor coefficients.  In effect we want
to treat $x$ as though it were a scalar, in line with our policy of always applying operators
coefficientwise.
With this in mind, we can rewrite \cref{eqn:fdb
coproduct} and \cref{eqn:rio coproduct} simply as
\[
    \Delta \Pi(x) = \sum_{j \ge 0} \Pi(x)^{j+1} \otimes \pi_j
\]
and
\[
    \Delta Y(x) = \sum_{j \ge 0} Y(x)\Pi(x)^j \otimes y_j.
\]
Our first lemma is a common generalization of both formulas.

\begin{lemma} \label{thm:general rio coproduct}
    For any $s \in \KK$ and $k \in \Nat$,
    \[
        \Delta\big(Y(x)^s\Pi(x)^k\big) = \sum_{j \ge 0} Y(x)^s\Pi(x)^j \otimes
        [x^j]Y(x)^s\Pi(x)^k.
    \]
\end{lemma}

(Note that since $\Pi(x)$ has zero constant term, we can raise it only to nonnegative integer powers
if we want to stay in the realm of power series.)

\begin{proof}
    Both sides are power series with coefficients that are polynomials in $s$, so it is sufficient
    to prove the case $s \in \Nat$. Then by the coproduct formulas we can write
    \begin{align*}
        \Delta\big(Y(x)^s\Pi(x)^k\big)
        &= \sum_{i_1, \dots, i_s, j_1, \dots, j_k} Y(x)^s \Pi(x)^{i_1 + \dots + i_s + j_1 + \dots +
        j_k + k} \otimes y_{i_1} \cdots y_{i_s} \pi_{j_1} \cdots \pi_{j_k} \\
        &= \sum_{j \ge 0} Y(x)^s \Pi(x)^j \otimes \sum_{i_1 + \dots + i_s + j_1 + \dots + j_k + k =
            j} y_{i_1} \cdots y_{i_s} \pi_{j_1} \cdots \pi_{j_k} \\
        &= \sum_{j \ge 0} Y(x)^s \Pi(x)^j \otimes [x^j]Y(x)^s \Pi(x)^k
    \end{align*}
    as desired.
\end{proof}

For $n \ge 0$, let $\FdB^{(n)}$ denote the subalgebra of $\FdB$ generated by $\pi_1, \dots,
\pi_{n-1}$ (this should not be confused with the graded piece $\FdB_n$) and let $\RioHopf^{(n)}$
denote the subalgebra of $\RioHopf$ generated by $\pi_1, \dots, \pi_{n-1}$ and $y_1, \dots, y_n$.
From the coproduct formulas it is clear that  are in fact sub-bialgebras. The following result is
new as stated but encapsulates the main calculation used in standard proofs of \cref{thm:dse single
rio}.

\begin{lemma} \label{thm:cocycle fdb implies rio}
    Suppose $H$ is a bialgebra, $\phi\colon \RioHopf \to H$ is an algebra morphism, and
    $\{\Lambda_k\}_{k \in \Natp}$ is a family of 1-cocycles on $H$. Let $\Phi(x) = \phi(\Pi(x))$ and
    suppose $\phi(Y(x)) = F(x)$ where $F(x)$ is the unique solution to
    \begin{equation} \label{eqn:fake dse with q}
        F(x) = 1 + \sum_{k \ge 1} \Lambda_k(F(x)\Phi(x)^k).
    \end{equation}
    Then for $n \ge 0$, if $\phi$ is a bialgebra morphism when restricted to $\FdB^{(n)}$, it is
    also a bialgebra morphism when restricted to $\RioHopf^{(n)}$.
\end{lemma}

Recall that by definition $\Pi(x)$ and hence also $\Phi(x)$ has zero constant term, so only the
terms with $k \le n$ on the right side of \cref{eqn:fake dse with q} can contribute to the
coefficient of $x^n$. Thus the equation really does have a unique solution.

\begin{proof}
    Since we are given that $\phi$ is an algebra morphism we must only prove it preserves the
    coproducts of the generators. We prove this by induction on $n$. In the base case, $\FdB^{(0)} =
    \RioHopf^{(0)} = \KK$ so there is nothing to prove. Now suppose that $n > 0$ and that $\phi$ is
    a bialgebra morphism when restricted to $\RioHopf^{(n-1)}$ and also preserves the coproduct of
    $\pi_{n-1}$. Then we must show it preserves the coproduct of $y_n$. Note that when $k > 1$, the
    coefficient $[x^n]Y(x)\Pi(x)^k$ does not contain $y_n$, so its coproduct agrees with the formula
    from \cref{thm:general rio coproduct}. Thus
    \begin{align*}
        \Delta \phi(y_n)
        &= \Delta([x^n]F(x)) \\
        &= \Delta\left(\sum_{k \ge 1} \Lambda_k\left([x^n]F(x)\Phi(x)^k\right)\right) \\
        &= \left(\Lambda_k \otimes 1 + (\id \otimes \Lambda_k)\Delta\right)\left(\sum_{k \ge 1}
        [x^n]F(x)\Phi(x)^k\right) \\
        &= [x^n]F(x) \otimes 1 + \sum_{k \ge 1} \sum_{j=0}^{n-k} [x^n] F(x)\Phi(x)^j \otimes
        [x^j] \Lambda_k(F(x)\Phi(x)^k) \\
        &= [x^n]F(x) \otimes 1 + \sum_{j=0}^{n-1} [x^n]F(x)\Phi(x)^j \otimes [x^j] \left(\sum_{k \ge
        1} \Lambda_k(F(x)\Phi(x)^k)\right) \\
        &= \sum_{j=0}^{n} [x^n]F(x)\Phi(x)^j \otimes [x^j]F(x) \\
        &= (\phi \otimes \phi)(\Delta y_n). \qedhere
    \end{align*}
\end{proof}

\begin{remark}
    An obvious consequence of \cref{thm:cocycle fdb implies rio} is that if $\phi$ is already known
    to be a bialgebra morphism when restricted to $\FdB$ then it is a bialgebra morphism on all of
    $\RioHopf$. This is not quite the right version of the statement for the application to
    Dyson--Schwinger equations, but it does give some interesting examples of series satisfying
    renormalization group equations. For instance, consider the map $\FdB \to \CK$ given by $\pi_n
    \mapsto \ell_1^n$. (Recall that $\ell_n$ is the $n$-vertex ladder so in particular $\ell_1$ is
    the unique one-vertex tree.) It is a straightforward exercise to show that this is in fact a
    bialgebra morphism. Thus we can extend this map to $\RioHopf$ by sending $Y(x)$ to the series
    $T(x)$ defined by
    \[
        T(x) = 1 + xB_+\left(\frac{T(x)}{1-\ell_1x}\right),
    \]
    an example due to Dugan \cite{dugan:masters} of a series not coming from a Dyson--Schwinger
    equation which nonetheless satisfies a renormalization group equation after applying a bialgebra
    morphism $\CK \to \KK[L]$. In the spirit of \cref{example:positive s,example:s = -1}, we can
    think of $T(x)$ as a generating function for plane trees with the property that one obtains a
    ladder after deleting all leaves.
\end{remark}

We can now prove \cref{thm:dse single rio}.

\begin{proof}[Proof of \cref{thm:dse single rio}]
    We prove by induction on $n$ that $\phi$ is a bialgebra morphism on $\RioHopf^{(n)}$. For $n =
    0$ this is trivial. Now suppose $n > 0$ and that $\phi$ is a bialgebra morphism on
    $\RioHopf^{(n-1)}$. In particular, $\phi$ is a bialgebra morphism on $\FdB^{(n-1)}$, and we
    observe that since $[x^{n-1}]T(x)^s \in \phi(\RioHopf^{(n-1)})$, by \cref{thm:general rio
    coproduct} we have
    \begin{align*}
        \Delta \phi(\pi_{n-1})
        &= \Delta([x^{n-1}]T(x)^s) \\
        &= \sum_{j \ge 0} [x^{n-1}] T(x)^s (xT(x)^s)^j \otimes [x^j]T(x)^s \\
        &= \sum_{j \ge 0} [x^{n-1}] (xT(x)^s)^{j+1} \otimes [x^j]T(x)^s \\
        &= \sum_{j \ge 0} [x^{n-1}] \phi(\Pi(x))^{j+1} \otimes \phi(\pi_j) \\
        &= (\phi \otimes \phi)(\Delta \pi_{n-1})
    \end{align*}
    so $\phi$ is a bialgebra morphism on $\FdB^{(n)}$. By \cref{thm:cocycle fdb implies rio}, $\phi$
    is thus a bialgebra morphism on $\RioHopf^{(n)}$ as wanted. The renormalization group equation
    then follows from \cref{thm:rio pde}.
\end{proof}

Now we consider systems. The idea is the same, that we would like to write each equation of the
system in a form that looks like \cref{eqn:fake dse with q}. In general this will only work if we
have the same invariant charge for each equation. In terms of the setup, for
$p \in P_i$ we want a linear relation
\[
    \mu_p = 1_i + w_p\mathbf s
\]
for some $\mathbf s = (s_i)_{i \in I} \in \KK^I$. As in the single-equation case, we may as well
combine terms together to write the system in the form
\begin{equation} \label{eqn:dse system with q}
    G_i(x, L) = 1 + \sum_{k \ge 1} x^k\Lambda_{i,k}\left(G_i(x,L)\prod_j G_j(x, L)^{s_jk}\right).
\end{equation}
The corresponding combinatorial system then looks like
\begin{equation} \label{eqn:cdse system with q}
    T_i(x) = 1 + \sum_{k \ge 1} B_+^{(i, k)}(T_i(x)Q(x)^k)
\end{equation}
where
\begin{equation} \label{eqn:definition of q}
    Q(x) = x \prod_{i \in I} T_i(x)^{s_i}.
\end{equation}
We then have the following generalization of \cref{thm:dse single rio}. (Note that most of the
papers referenced in \cref{rem:nothing new under the delta} are actually for this version already.)

\begin{theorem} \label{thm:dse system rio}
    Let $\mathbf T(x) \in \CK_{I \times \Natp}[[x]]^I$ be the solution to the combinatorial
    Dyson--Schwinger system \cref{eqn:cdse system with q}. Then for any $i \in I$, the map
    $\phi_i\colon \RioHopf \to \CK_{I \times \Natp}$ defined by $\phi_i(Y(x)) = T_i(x)$ and
    $\phi_i(\Pi(x)) = Q(x)$ is a bialgebra morphism. As a consequence, the solution $\mathbf G(x, L)$
    to the corresponding Dyson--Schwinger system \cref{eqn:dse system with q} satisfies the
    renormalization group equations
    \[
        \left(\frac{\partial}{\partial L} - \beta(x) \frac{\partial}{\partial x} -
        \gamma_i(x)\right) G_i(x, L) = 0
    \]
    where $\gamma_i(x)$ is the linear term in $L$ of $G_i(x, L)$ and
    \[
        \beta(x) = \sum_{i \in I} s_ix\gamma(x).
    \]
\end{theorem}

\begin{proof}
    We prove by
    induction on $n$ that $\phi_i$ is a bialgebra morphism on $\RioHopf^{(n)}$ for all $i$.
    Supposing they are all bialgebra morphisms on $\RioHopf^{(n-1)}$. Then, as in the proof of
    \cref{thm:dse single rio}, we have
    \begin{align*}
        \Delta \phi_i(\pi_{n-1})
        &= \Delta([x^{n-1}] Q(x)) \\
        &= [x^{n-1}] \prod_{i \in I} \Delta(T_i(x)^{s_i}) \\
        &= [x^{n-1}] \sum_{\alpha \in \Nat^I} \prod_{i \in I} \left(T_i(x)^{s_i}Q(x)^{\alpha_i} \otimes
        [x^{\alpha_i}] T_i(x)^{s_i}\right) \\
        &= \sum_{\alpha \in \Nat^I} [x^n] Q(x)^{|\alpha| + 1} \otimes \prod_{i \in I}
        [x^{\alpha_i}]T_i(x)^{s_i} \\
        &= \sum_{j \ge 0} [x^n] Q(x)^{j+1} \otimes [x^{j+1}]Q(x) \\
        &= (\phi \otimes \phi)(\Delta \pi_{n-1})
    \end{align*}
    and thus $\phi_i$ is a bialgebra morphism on $\FdB^{(n)}$ and hence on $\RioHopf^{(n)}$ by
    \cref{thm:cocycle fdb implies rio}. The renormalization group equation then follows from
    \cref{thm:rio pde}.
\end{proof}

\subsection{The anomalous dimension revisited} \label{sec:equations for gamma}

Consider the Dyson--Schwinger system \cref{eqn:dse system with q}. Write the 1-cocycles in the form
given by \cref{thm:polynomial cocycles}:
\[
    \Lambda_{i,k} f(L) = \int_0^L A_{i,k}(d/du)f(u)\,du.
\]
We can then differentiate both sides of \cref{eqn:dse system with q} with respect to $L$:
\begin{equation} \label{eqn:differentiate dse}
    \frac{\partial}{\partial L} G_i(x, L) = \sum_{k \ge 1} x^k A_{i,k}(\partial/\partial L) G_i(x,
    L)\left(\prod_j G_j(x, L)^{s_j}\right)^k.
\end{equation}
For the moment writing $\Psi(x,L) =\prod_j G_j(x,L)^{s_j}$ and using \cref{thm:fdb pde}, \cref{thm:rio pde}, and \cref{thm:dse system rio}, we have that
\[
\frac{\partial \Psi(x,L)}{\partial L} = \beta(x)\frac{\partial \Psi(x,L)}{\partial x}
\]
and hence
\begin{align*}
    & \frac{\partial}{\partial L} \left(G_i(x, L) \Psi(x,L)^k\right) \\
   & = \left( \left(\beta(x)\frac{\partial}{\partial x} + \gamma_i(x)\right)G_i(x,L)\right)\Psi(x,L)^k + \beta(x)G_i(x,L)\Psi(x,L)^{k-1}\left(\frac{\partial}{\partial x}\Psi(x,L)\right) \\
   & = \left(\beta(x)\frac{\partial}{\partial x} + \gamma_i(x)\right)G_i(x,L)\Psi(x,L)
\end{align*}
Putting this back into \cref{eqn:differentiate dse}, we obtain
\begin{equation} \label{eqn:combine dse with rge}
    \frac{\partial}{\partial L} G_i(x, L) = \sum_{k \ge 1}
    A_{i,k}\left(\beta(x)\frac{\partial}{\partial x} + \gamma_i(x)\right) G_i(x, L)\left(\prod_j
    G_j(x, L)^{s_j}\right)^k.
\end{equation}
If we then set $L = 0$, we obtain the following somewhat strange equation satisfied by the anomalous
dimension:
\begin{equation} \label{eqn:equation for gamma}
    \gamma_i(x) = \sum_{k \ge 1} A_{i,k}\left(\beta(x)\frac{\partial}{\partial x} +
    \gamma_i(x)\right) x^k.
\end{equation}

Two special cases are interesting. Firstly, in the case of only a single term $k = 1$,
\cref{eqn:equation for gamma} can be rewritten as a pseudo-differential equation
\begin{equation} \label{eqn:equation for gamma differential}
    A_i\left(\beta(x)\frac{\partial}{\partial x} + \gamma_i(x)\right)^{-1} \gamma_i(x) = x.
\end{equation}
This formulation is due to Balduf \cite[Equation 2.19]{balduf:subtraction} but special cases have
been known for longer \cite{bk01,yeats:phd,bellon10}. If $A_i(x)$ is the reciprocal of a polynomial,
which occurs in some physically relevant cases, \cref{eqn:equation for gamma differential} becomes
an honest differential equation. These equations have been analyzed by Borinsky, Dunne, and the
second author \cite{bdy} using the combinatorics of tubings.

The other special case is that of a \textit{linear} system, i.e. $\beta(x) = 0$. In this case
\cref{eqn:equation for gamma} becomes a functional equation
\begin{equation} \label{eqn:equation for gamma functional}
    \gamma_i(x) = \sum_{k \ge 1} x^k A_{i,k}(\gamma_i(x)).
\end{equation}
In the intersection of these two special cases, a linear equation with a single term,
\cref{eqn:equation for gamma functional} can be solved explicitly by Lagrange inversion;
see \cite[Lemma 2.32]{tubings}.

\subsection{Tubings}\label{subsec tubings}

We now introduce the combinatorial objects that we will use to understand 1-cocycles. We will only
be interested in trees, but the basic definitions can be given in the context of an arbitrary finite
poset $P$. A \textit{tube} is a connected convex subset of $P$.
For tubes $X, Y$ write $X \to Y$ if $X \cap Y = \emptyset$ and there exist $x \in X$ and $y \in Y$
such that $x < y$.  A collection $\tau$ of tubes is called a \textit{tubing} if it satisfies the
following conditions:
\begin{itemize}
    \item
        (Laminarity) If $X, Y \in \tau$ then either $X \cap Y = \emptyset$, $X \subseteq Y$, or $X
        \supseteq Y$.

    \item
        (Acyclicity) There do not exist tubes $X_1, \dots, X_k \in \tau$ with $X_1 \to X_2 \to \dots
        \to X_k \to X_1$.
\end{itemize}

Tubings of posets (also called \textit{pipings}) were introduced by Galashin \cite{galashin:tubings}
to index the vertices of a certain polytope associated to $P$, the \textit{$P$-associahedron}. They
were rediscovered (in the case of trees) by the authors of \cite{tubings} in the present context.
Note that for trees the acyclicity condition is trivial.

\begin{remark}
    Galashin defines a \textit{proper tube} to be one which is neither a singleton nor the entirety
    of $P$, and a \textit{proper tubing} to be one consisting only of proper tubes. Only the proper
    tubes and tubings play a role in the definition of the poset associahedron, but for us it will
    be sensible to include the improper ones. Note that if one restricts attention to
    \textit{maximal} tubings (which we largely will do) then this makes no combinatorial difference,
    as a maximal tubing contains \textit{all} of the improper tubes and removing them maps the set
    of maximal tubings bijectively to the set of maximal proper tubings.
\end{remark}

\begin{remark}
    Tubings of posets are only loosely related to the better-known notion of tubings of graphs
    introduced by Carr and Devadoss \cite{carr-devadoss}. For graphs, a \textit{tube} is defined to
    be a set of vertices which induces a connected subgraph and a \textit{tubing} is a set of tubes
    satisfying the laminarity condition along with a certain \textit{non-adjacency} condition which
    is entirely different from the acyclicity condition for poset tubings.  Thus the notions should
    not be confused. However, in the case of trees there is a close connection: tubings of a rooted
    tree (as a poset) are in bijection with tubings of the \textit{line graph} of the tree. (This is
    essentially a special case of a result of Ward \cite[Lemma 3.17]{ward21} which is stated in
    terms of related objects called \textit{nestings}. See \cite[Section 6]{tubings} for a
    discussion of this in our language.)
\end{remark}

\begin{remark}
    A subset of a rooted tree is convex and connected (in the order-theoretic sense) if and only if
    it is connected in the graph-theoretic sense. Thus the set of tubings of a rooted tree is really
    an invariant of the underlying unrooted tree. However, the statistics on tubings in which we
    will be interested do depend on the root and are best thought of in terms of the partial
    ordering.
\end{remark}

The laminarity condition implies that if $\tau$ is a tubing of $P$, the poset of tubes ordered by
inclusion is a forest. In the case of a maximal tubing of a connected poset, there is a unique
maximal tube (namely $P$ itself) and each non-singleton tube has exactly two maximal tubes properly
contained within it. Relative to $X$, one of these tubes is a downset and one is an upset. Taking
the downset as the left child and the upset as the right child, the tubes of a maximal tubing thus
have the structure of a binary plane tree. For this reason (and to avoid confusion with graph
tubings), maximal tubings of rooted trees were referred to as \textit{binary tubings} in
\cite{tubings} and we will also use this language.

\begin{figure}
    \begin{center}
        \subcaptionbox[]{\label{fig:tubing example small}All binary tubings of a small
        tree.}{
            \includegraphics[page=4,width=0.6\textwidth]{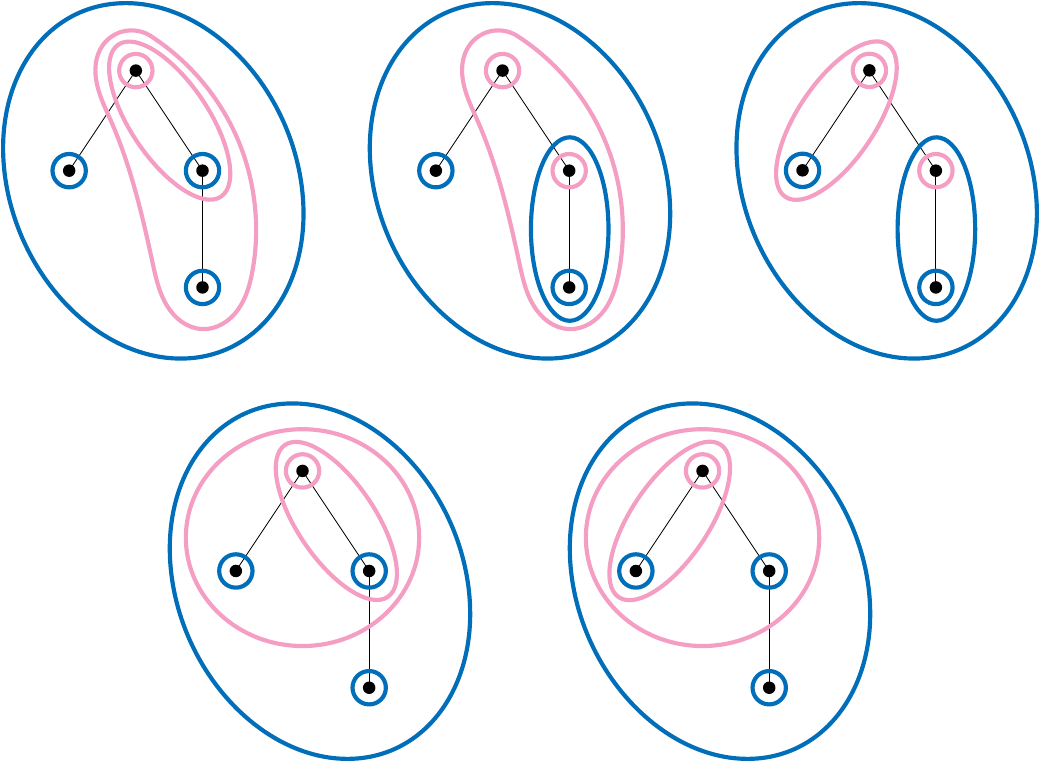}
        }
        \subcaptionbox[]{\label{fig:tubing example large}A binary tubing of a larger
        tree.}{
            \includegraphics[page=2,width=0.3\textwidth]{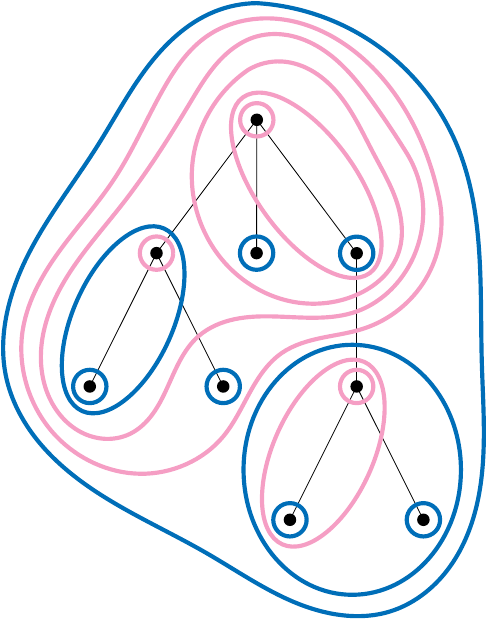}
        }
    \end{center}
    \caption{Examples of binary tubings. Upper and lower tubes highlighted in different colours.}
\end{figure}

\label{symboldef:tub}\label{symboldef:rank}\label{symboldef:bstat}Henceforth we shall restrict our
attention exclusively to binary tubings of rooted trees. We will write $\Tub(t)$ for the set of
binary tubings of $t$. We will call a tube a \textit{lower tube} (resp. \textit{upper tube}) if it
is a downset (resp. upset) in its parent in the tree of tubes. We will also consider $t$ itself to
be a lower tube; this ensures that each vertex is the root of exactly one lower tube. Given a vertex
$v$, define the \textit{rank} $\rank(\tau, v)$ of $v$ in $\tau$ to be the number of upper tubes
rooted at $v$.\footnote{In \cite{tubings} the equivalent statistic $b(\tau, v) = \rank(\tau, v) + 1$
counting the total number of tubes rooted at $v$ was used instead. It has since become clear that
the rank is really the fundamental quantity.} We will write $b(\tau)$ for the number of tubes of
$\tau$ containing the root of $t$; note that we clearly have $b(\tau) = \rank(\tau, \rt t) + 1$.

\begin{remark} \label{rem:rank in terms of lower}
    Our definition of the rank refers only to the upper tubes, but it can be equivalently defined in
    terms of \textit{lower} tubes only: for each upper tube rooted at $v$ there is a corresponding
    lower tube, with the property that there is no lower tube of $\tau$ lying strictly between it
    and the unique lower tube rooted at $v$. Conversely any such lower tube corresponds to an upper
    tube rooted at $v$. In other words, considering the lower tubes of $\tau$ as a poset ordered by
    containment, $\rank(\tau, v)$ is the number of lower tubes that are covered by the unique lower
    tube rooted at $v$.
\end{remark}

Binary tubings of rooted trees have a recursive structure which we will make essential use of.

\begin{proposition} \label{thm:recursive tubing}
    Let $t$ be a rooted tree with $|t| > 1$. There is a bijection between binary tubings of $t$ and
    triples $(t', \tau', \tau'')$ where $t'$ is a proper subtree (principal downset) of $t$, $\tau'
    \in \Tub(t')$, and $\tau'' \in \Tub(t \setminus t')$. Moreover this bijection satisfies
    \[
        \rank(\tau, v) = \begin{cases}
            \rank(\tau', v) & v \in t' \\
            \rank(\tau'', v) + 1 & v = \rt t \\
            \rank(\tau'', v) & \text{otherwise}
        \end{cases}
    \]
    and $b(\tau) = b(\tau'') + 1$.
\end{proposition}

\begin{proof}
    By the discussion above there are two maximal tubes $t', t''$ properly contained in the largest
    tube $t$, where $t'$ is a downset and $t'' = t \setminus t'$ is an upset and both are connected.
    A connected downset in a rooted tree is a subtree; since the complement is nonempty it must be
    that $t'$ is a proper subtree. Let
    \begin{align*}
        \tau' &= \{u \in \tau\colon u \subseteq t'\}
        \intertext{and}
        \tau'' &= \{u \in \tau\colon u \cap t' = \emptyset\}.
    \end{align*}
    By the laminarity condition, we have $\tau = \{t\} \cup \tau' \cup \tau''$. Since each tube in
    $\tau'$ and $\tau''$ still contains two maximal tubes within it, these are still maximal tubings
    of $t'$ and $t''$ respectively, i.e. $\tau' \in \Tub(t')$ and $\tau'' \in \Tub(t \setminus t')$.
    Note that the upper tubes of $\tau'$ and $\tau''$ are also upper tubes of $\tau$, whereas $t''$
    is a lower tube in $\tau''$ and an upper tube in $\tau$. Thus the statement about ranks follows.
    Since $b(\tau) = \rank(\tau, \rt t) + 1$ and $\rt t \in t''$ the statement about the
    $b$-statistic also follows.
\end{proof}

\begin{remark}
    Observe that in a binary tubing we split the tree into an upper and lower part just as in the
    definition of the coproduct, but with the key difference that they are both required to be
    trees. To make this more algebraic, let $P_\mathrm{lin}\colon \CK \to \CK$ be the projection
    onto the subspace spanned by trees.  Then the \textit{linearized coproduct} is
    $\Delta_\mathrm{lin} = (P_\mathrm{lin} \otimes P_\mathrm{lin})\Delta$. In effect,
    $\Delta_\mathrm{lin}$ looks the same as the usual coproduct but only includes terms where both
    tensor factors are trees rather than arbitrary forests. Unlike the coproduct, the linearized
    coproduct fails to be coassociative: there are multiple different maps $\CK \to \CK^{\otimes k}$
    that can be built by iterating it. For instance in the case $k = 3$ there are distinct maps
    $(\Delta_\mathrm{lin} \otimes \id)\Delta_\mathrm{lin}$ and $(\id \otimes
    \Delta_\mathrm{lin})\Delta_\mathrm{lin}$. It is not hard to see that if we iterate all the way
    to $k = |t|$, the terms that we can get from all of these maps taken collectively correspond
    exactly to the binary tubings.

    The linearized coproduct is itself a nice algebraic object as it is co-pre-Lie and specifically is dual to the pre-Lie product given by insertion of rooted trees.  See \cite{CLpre} for more on the structure of this product.
\end{remark}

We are now ready to give the tubing expression for maps $\CK_I \to \KK[L]$ induced by the universal
property \cref{thm:ck universal decorated} which was the main result of \cite{tubings}. Let us fix a set $I$ of decorations and a family of
1-cocycles
\[
    \Lambda_i f(L) = \int_0^L A_i(d/du)f(u)\,du
\]
where
\[
    A_i(L) = \sum_{n \ge 0} a_{i,n} L^n.
\]
Given an $I$-tree $t$, let us write $d(t)$ for the decoration of the root vertex, and $d(v)$ for the
decoration of a vertex $v$. \label{symboldef:mel}For a tubing $\tau$ of $t$, we define the
\textit{Mellin monomial} 
\[
    \mel(\tau) = \prod_{\substack{v \in t \\ v \ne \rt t}} a_{d(v), \rank(\tau, v)}.
\]
We call this the Mellin monomial since it is a monomial made of coefficients from the Mellin transforms of the primitives driving the Dyson--Schwinger equation. 
With these definitions in hand, we can state the formula.

\begin{theorem}[{\cite[Theorem 4.2]{tubings}}] \label{thm:tubing expansion 1}
    With the above setup, the unique map $\phi\colon \CK_I \to \KK[L]$ satisfying $\phi B_+^{(i)} =
    \Lambda_i \phi$ is given on trees by the formula
    \[
        \phi(t) = \sum_{\tau \in \Tub(t)} \mel(\tau) \sum_{k=1}^{b(\tau)} a_{d(t),
        b(\tau) - k} \frac{L^k}{k!}.
    \]
\end{theorem}

\begin{example}
    Let $t$ be the tree that appears in \cref{fig:tubing example small}. Computing the contributions
    of the five tubings and summing them up, we get
    \begin{align*}
        \phi(t) ={} &a_0^3\left(a_3L + a_2 \frac{L^2}{2!} + a_1 \frac{L^3}{3!} + a_0
        \frac{L^4}{4!}\right)
        + 2a_0^2a_1\left(a_2L + a_1 \frac{L^2}{2!} + a_0 \frac{L^3}{3!}\right) \\
      &+ 2a_0^3\left(a_2L + a_1 \frac{L^2}{2!} + a_0 \frac{L^3}{3!}\right)
    \end{align*}
    where the second and third tubing in the figure give the same contribution, as do the fourth and
    fifth. These coincidences can be explained combinatorially by the fact that in both cases the
    offending pair of tubings differ in the upper tubes but have the exact same set of lower tubes,
    which in light of \cref{rem:rank in terms of lower} is sufficient to determine the Mellin
    monomial and $b$-statistic.
\end{example}

Combining \cref{thm:cdse system solution} and \cref{thm:tubing expansion 1} we can solve the single insertion place Dyson--Schwinger equations and systems thereof.
 First, we will solve the combinatorial Dyson--Schwinger equation giving a series in $T(x) \in \CK_P[[x]]$ which encodes the recursive
structure of the DSE. We then apply the unique map $\phi\colon \CK_P \to \KK[L]$ satisfying $\phi
B_+^{(p)} = \Lambda_p \phi$ which exists by \cref{thm:ck universal decorated} to get $G(x, L) =
\phi(T(x))$ which will solve the Dyson--Schwinger equation.  The combinatorial expansion for $\phi$ (\cref{thm:tubing expansion 1}) then gives a combinatorial expansion for the solution of the Dyson--Schwinger equation.

Specifically, we obtain the following theorem.

\begin{theorem}[{\cite[Theorem 2.12]{tubings}}] \label{thm:dse tubing expansion 1}
    The unique solution to \cref{eq:dse lambda single} is
    \[
        G(x, L) = 1 + \sum_{t \in \Trees(P)}\left(\prod_{v \in t} \mu(v)^{\underline{\od(v)}}
        \right) \sum_{\tau \in \Tub(t)} \mel(\tau) \sum_{k=1}^{b(\tau)} a_{d(t), b(\tau) - k}
        \frac{x^{w(t)}L^k}{|\Aut(t)|k!}.
    \]
    In particular, the solution to the special case \cref{eq:dse lambda single simple} is
    \[
        G(x, L) = 1 + \sum_{t \in \Trees(\Natp)}\left(\prod_{v \in t} (1 +
        sw(v))^{\underline{\od(v)}} \right) \sum_{\tau \in \Tub(t)} \mel(\tau) \sum_{k=1}^{b(\tau)}
        a_{w(\rt t), b(\tau) - k} \frac{x^{w(t)}L^k}{|\Aut(t)|k!}.
    \]
\end{theorem}

Similarly for systems of Dyson--Schwinger equations we have the following result.
\begin{theorem}[{\cite[Theorem 4.8]{tubings}}] \label{thm:dse tubing expansion 1 system}
    The unique solution to the system \cref{eq:dse lambda system} is
    \[
        G_i(x, L) = 1 + \sum_{t \in \Trees_i(P)} \left(\prod_{v \in t} \mu(v)^{\underline{\odv(v)}}
        \right) \sum_{\tau \in \Tub(t)} \mel(\tau) \sum_{k=1}^{b(\tau)} a_{d(t), b(\tau) - k}
        \frac{x^{w(t)}L^k}{|\Aut(t)|}.
    \]
\end{theorem}

\begin{remark}
    The tubing expansions can be related to the usual Feynman diagram expansions of Dyson--Schwinger equations by taking the trees that appear in these expansions as \textit{insertion trees} which encode
    the way a Feynman diagram is built from primitive diagrams. One can in principle recover the
    contribution of an individual Feynman diagram by an appropriately weighted sum over (tubings of)
    insertion trees for that diagram. (See for instance \cite{kreimer99}).  A more detailed discussion of how the tubing expansion relates to the Feynman diagrams can be found in Remark 5.25 of \cite{tubings}.
\end{remark}

\section{Multiple insertion places}

\subsection{Context}\label{subsec context}

With the background in hand, we are ready to move to considering multiple insertion places.  Before proving the results, we should return to the set up to understand what multiple insertion places means.

For the moment we want to view the Dyson--Schwinger equations diagrammatically as equations telling us how to insert primitive Feynman diagrams into each other to obtain series of Feynman diagrams (then applying Feynman rules will bring us to the $G(x,L)$ we've been working with).  To keep things simple for the intuition (though the general picture is much the same) let's consider the case with just one 1-cocyle $\Lambda$ in \cref{eq:dse lambda single} and with $w = 1, \mu = -1$.  Consider the way the Mellin transform $A(\rho)$ comes into the Dyson--Schwinger equation when the 1-cocycle is rewritten according to \cref{remark:differential cocycle form}.  $A(\rho)$ is essentially the Feynman integral of the primitive with $\rho$ acting as a regulator on the edge into which we are inserting (see for instance \cite{yeats:phd} for a derivation given in close to this language).  In particular, this implies that all the insertions are into the same edge.  This is, in fact, easiest to see with the Dyson--Schwinger equation in its integral equation form.  For instance, a particular instantiation of the case presently at hand in a massless Yukawa theory, and phrased in a compatible language to what we're using here can be found at the end of Section 3.2 of \cite{yeats:phd}.  In the Dyson--Schwinger equation in this form the way we see that the insertion is in a single propagator is by the scale variable of the inserted $G(x,\log k^2)$s inside the integral being the log of the momentum squared of one particular propagator---the one into which we are inserting.

Let us return to the case with multiple 1-cocycles, hence with multiple primitive Feynman diagrams.
If we are willing to insert symmetrically into all insertion places, we can get around the fact that
our single variable Mellin transforms only accounts for a single insertion place (see Section 2.3.3
of \cite{yeats:phd}), but for a finer understanding, we would like to be able to work with multiple
insertions places more honestly. In particular this means that our Mellin transforms of our
primitives should be regularized on all the edges with different variables on each, giving
$A(\rho_1, \rho_2, \ldots, \rho_m)$. Then the recursive appearances of $G(x, L)$ or of the
$G_i(x,L)$ should correspond to different $\rho_i$ according to the type of edge in the Feynman
diagram.  With the 1-cocycles written in the form of \cref{remark:differential cocycle form}, this
means we are interested in equations of the form
\begin{multline} \label{eqn:original multidse}
    G(x, L) = \\
    1 + x G(x, \partial/\partial \rho_1)^{-1} \cdots G(x, \partial/\partial \rho_m)^{-1}
     (e^{L(\rho_1 + \dots + \rho_m)} - 1) F(\rho_1, \dots, \rho_m)\big|_{\rho=0}
\end{multline}
where each edge has its own variable in the Mellin transform, as well as similar systems of equations.  Such equations have been set up in language similar to this by one of us \cite{yeats:phd} and considered further by Nabergall \cite[Section 4.2]{nabergall:phd} but until now we had no combinatorial handle of the sort given by the tubing expansions above.
The goal of this paper is to give tubing expansion solutions to such equations, thus giving combinatorial solutions to all single scale Dyson--Schwinger equations.

Mathematically, the set up is as follows.
We again have a set $P$ which will index our cocycles, but
to each $p \in P$ we associate a finite set $E_p$ of \textit{insertion places}. Each insertion place
$e$ has its own insertion exponent $\mu_{e}$; sometimes it will still be convenient to refer to
the \textit{overall insertion exponent}
\[
    \mu_p = \sum_{e \in E_p} \mu_{e}.
\]
Finally, to each $p$ we associate a vector of indeterminates $\mathbf L_p = (L_e)_{e \in E_p}$ and a
1-cocycle $\Lambda_p \in \Cycles^1(\KK[L], \KK[\mathbf L_p])$. The Dyson--Schwinger associated to
these data is
\begin{equation} \label{eqn:multidse single}
    G(x, L) = 1 + \sum_{p \in P} x^{w_p} \Lambda_p\left(\prod_{e \in E_p} G(x, L_e)^{\mu_e}\right).
\end{equation}

We will also consider systems. As before we partition our index set $P$ into
$\{P_i\}_{i \in I}$ and replace the insertion exponents with insertion exponent vectors. Our system
is then
\begin{equation} \label{eqn:multidse system}
    G_i(x, L) = 1 + \sum_{p \in P_i} x^{w_p} \Lambda_p\left(\prod_{e \in E_p} \mathbf G(x,
    L_e)^{\mu_{e}}\right).
\end{equation}

\subsection{1-cocycles and tensor powers}\label{subsec tensor 1-cocycles}
We next need to upgrade our results on 1-cocycles to tensor powers of a bialgebra $H$ as this will be the correct algebraic structure to account for multiple insertion places.

Note that there is a canonical comodule homomorphism $\mu\colon H^{\otimes r} \to H$, namely the
multiplication map $\mu(h_1 \cdots h_r) = h_1 \cdots h_r$. By \cref{thm:cocycle contravariant} we
can build various 1-cocycles $\Lambda\mu \in \Cycles^1(H, H^{\otimes r})$ for various $\Lambda \in
\Cycles^1(H)$. We will call such cocycles \textit{boring}; as we will see, they are generally the
trivial case for our results.

As in \cref{subsec hopf poly} we focus on the case of $H = \KK[L]$. We identify $\KK[L]^{\otimes r}$ with $\KK[L_1, \dots, L_r]$ made into a
comodule with coaction
\[
    \delta \mathbf L^\alpha = \sum_{\beta \le \alpha} \binom{\alpha}{\beta} L^{|\alpha| - |\beta|}
    \otimes \mathbf L^\beta.
\]
We now give a generalization of \cref{thm:polynomial cocycles} to tensor powers.

\begin{theorem} \label{thm:polynomial cocycles tensor}
    For any series $A(\mathbf L) \in \KK[[L_1, \dots, L_r]]$, the map $\KK[L_1, \dots, L_r] \to
    \KK[L]$ given by
    \begin{equation} \label{eqn:1-cocycle tensor integral}
        f(\mathbf L) \mapsto \int_0^L A(\partial/\partial u_1, \dots, \partial/\partial u_r)
        f(u_1, \dots, u_r) \big|_{u_1 = \dots = u_r = u}\,du
    \end{equation}
    is a 1-cocycle. Moreover, all 1-cocycles $\KK[L_1, \dots, L_r] \to \KK[L]$ are of this form.
\end{theorem}

\begin{proof}
    Let $\psi\colon \KK[L_1, \dots, L_r] \to \KK$ be given by
    \[
        \psi(\mathbf L^\alpha) = [\mathbf L^\alpha] A(\mathbf L).
    \]
    Then the operator defined by \cref{eqn:1-cocycle tensor integral} is simply $\Int \actoncocycle
    \psi$ where $\Int$ is the usual integral cocycle on $\KK[L]$ (see \cref{example:integral
    cocycle}) and the $\actoncocycle$ notation was defined in \cref{subsec hopf poly}. Thus by
    \cref{thm:cocycle contravariant} and \cref{thm:actoncocycle}, this operator is indeed a
    1-cocycle.

    Conversely, suppose $\Lambda$ is a 1-cocycle. We wish to find a series $A(\mathbf L)$ such that
    $\Lambda$ has the form \cref{eqn:1-cocycle tensor integral}. For $\alpha \in \Nat^r$ take
    $a_\alpha = \lin \Lambda(\mathbf L^\alpha)$ and let $A(\mathbf L)$ be the exponential generating
    function for these:
    \[
        A(\mathbf L) = \sum_{\alpha \in \Nat^r} \frac{a_\alpha \mathbf L^\alpha}{\alpha_1! \cdots
        \alpha_r!}.
    \]
    Now observe that for a polynomial $f(L)$ we have $\frac{df(L)}{dL} = \lin \leftact f(L)$.  With this in mind,
    \begin{align*}
        \frac{d \Lambda(\mathbf L^\alpha)}{dL}
        &= \lin \leftact \Lambda(\mathbf L^\alpha) \\
        &= (\id \otimes \lin)(\Lambda \otimes 1 + (\id \otimes \Lambda)\delta)\mathbf L^\alpha \\
        &= (\id \otimes \lin \Lambda)\delta \mathbf L^\alpha \\
        &= \sum_{\beta \le \alpha} \binom{\alpha}{\beta} a_\beta L^{|\alpha| - |\beta|} \\
        &= \sum_{\beta \le \alpha} \frac{a_\beta}{\beta_1! \cdots \beta_r!} \prod_{i=1}^r
        \frac{d^{\beta_i}L^{\alpha_i}}{dL^{\beta_i}} \\
        &= \sum_{\beta \le \alpha} \frac{a_\beta}{\beta_1! \cdots \beta_r!} \frac{\partial^{|\beta|}
        \mathbf u^\alpha}{\partial u_1^{\beta_1} \cdots \partial u_r^{\beta_r}} \bigg|_{u_1 = \dots
        = u_r = L}\\
        &= A\left(\frac{\partial}{\partial u_1}, \dots, \frac{\partial}{\partial u_r}\right) \mathbf
        u^\alpha \bigg|_{u_1 = \dots = u_r = L}
    \end{align*}
    and hence by linearity, for any polynomial $f(\mathbf L)$ we have
    \[
        \frac{d \Lambda f(\mathbf L)}{dL} = A\left(\frac{\partial}{\partial u_1}, \dots,
        \frac{\partial}{\partial u_r}\right) f(\mathbf u) \bigg|_{u_1 = \dots = u_r = L}
    \]
    which is also the derivative of the right-hand side of \cref{eqn:1-cocycle tensor integral}.
    But since $\Lambda f(\mathbf L)$ must have zero constant term by \cref{thm:cocycle counit}, it
    is exactly given by \cref{eqn:1-cocycle tensor integral}, as wanted.
\end{proof}

\begin{remark} \label{remark:boring polynomial cocycles}
    The multiplication map $\KK[L]^{\otimes r} \to \KK[L]$ corresponds to the map $\KK[\mathbf L]
    \to \KK[L]$ that substitutes $L$ for all of the variables. The adjoint map $\KK[[L]] \to
    \KK[[\mathbf L]]$ is the substitution $L \mapsto L_1 + \dots + L_r$. Thus the boring 1-cocycles
    correspond to series that expand in powers of the sum $L_1 + \dots + L_r$.
\end{remark}

\label{symboldef:cke}Next we construct the analogue in this setting of the Connes--Kreimer Hopf
algebra. Let $\widetilde{\mathcal T}_r$ be the set of unlabelled rooted trees with edges decorated
by elements of $\{1, \dots, r\}$ and $\widetilde{\mathcal F}_r$ the corresponding set of forests.
Define $\CKE_r$ to be the free vector space on $\widetilde{\mathcal F}_r$, made into an algebra with
disjoint union as multiplication and a downset/upset coproduct exactly as in $\CK$ but preserving the
decorations on the (remaining) edges. Now define $\tilde B_+\colon \CKE_r^{\otimes r} \to \CKE_r$ as
follows: for $f_1, \dots, f_r \in \widetilde{\mathcal F}_r$, $\tilde B_+(f_1 \otimes \dots \otimes
f_r)$ is the tree obtained from the forest $f_1 \cdots f_r$ by adding a new root with an edge to the
root of each component, where the edges to $f_i$ have decoration $i$. Note that clearly $\CKE_1
\cong \CK$ and in this case $\tilde B_+$ is just the usual $B_+$.

Very similar combinatorial and algebraic structures are used in the study of regularity structures, see \cite{BHZregularity}.

\begin{proposition} \label{thm:tildebplus is 1-cocycle}
    $\tilde B_+ \in \Cycles^1(\CKE_r, \CKE_r^{\otimes r})$.
\end{proposition}

\begin{proof}
    Let $t = \tilde B_+(f_1 \otimes \dots \otimes f_r)$. The only downset in $t$ that contains the
    root is all of $t$, and each other downset is the union of a downset in $f_i$ for each $i$. In
    the complementary upset, all edges on the root have the same decoration as they do in $t$, so
    \begin{align*}
        \Delta t
        &= \sum_{f \in J(t)} f \otimes (t \setminus f) \\
        &= 1 \otimes t + \sum_{f_1' \in J(f_1)} \cdots \sum_{f_r' \in J(f_r)} f_1' \cdots f_r'
        \otimes \tilde B_+((f_1 \setminus f_1') \otimes \dots \otimes (f_r \setminus f_r')).
    \end{align*}
    On the other hand, the left coaction $\delta$ of $\CKE_r$ on $\CKE_r^{\otimes r}$ is, by
    definition,
    \[
        \delta(f_1 \otimes \dots \otimes f_r) = \sum_{f_r' \in J(f_r)} f_1' \cdots f_r'
        \otimes (f_1 \setminus f_1') \otimes \dots \otimes (f_r \setminus f_r').
    \]
    Comparing these, we see that indeed
    \[
        \Delta \tilde B_+ = 1 \otimes \tilde B_+ + (\id \otimes \tilde B_+)\delta. \qedhere
    \]
\end{proof}

The universal property of $\CK$ naturally extends to $\CKE_r$. The proof is essentially the same as
the original Connes--Kreimer universal property result.

\begin{theorem} \label{thm:ck universal multilinear 1}
    Let $A$ be a commutative algebra and $\Lambda\colon A^{\otimes r} \to A$ a linear map. There
    exists a unique map $\phi\colon \CKE_r \to A$ such that $\phi \tilde B_+ = \Lambda \phi^{\otimes
    r}$. Moreover, if $A$ is a bialgebra and $\Lambda$ is a 1-cocycle then $\phi$ is a bialgebra
    morphism.
\end{theorem}

\begin{proof}
    Suppose $t \in \widetilde{\mathcal T}_r$. We can uniquely write $t$ in the form $t = \tilde
    B_+(f_1 \otimes \dots \otimes f_r)$. Then we can recursively set
    \[
        \phi(t) = \Lambda(\phi(f_1) \otimes \dots \otimes \phi(f_r))
    \]
    where for a forest, $\phi(f)$ is the product over the components. Clearly this is a well-defined
    algebra map and the unique one satisfying the desired identity.

    If $A$ is a bialgebra and $\Lambda$ is a 1-cocycle, we compute
    \begin{align*}
        \Delta \phi(t)
        &= \Delta \Lambda(\phi(f_1) \otimes \dots \otimes \phi(f_r)) \\
        &= \phi(t) \otimes 1 + (\id \otimes \Lambda) \delta(\phi(f_1) \otimes \dots \otimes
        \phi(f_r))
    \end{align*}
    where $\delta$ is the coaction of $A$ on $A^{\otimes r}$. Now suppose that $\phi$ preserves
    coproducts for each of the forests, i.e.
    \[
        \Delta \phi(f_i) = \sum_{f_i' \in J(f_i)} \phi(f_i') \otimes \phi(f_i \setminus f_i').
    \]
    Then
    \begin{align*}
        \delta(\phi(f_1) \otimes \dots \otimes \phi(f_r))
        &= \sum_{f_1' \in J(f_1)} \cdots \sum_{f_r' \in J(f_r)} \phi(f_1' \cdots f_r') \otimes
        \phi(f_1 \setminus f_1') \otimes \dots \otimes \phi(f_r \setminus f_r')
    \end{align*}
    and hence
    \begin{align*}
        \Delta \phi(t)
        &= \phi(t) \otimes 1 + \sum_{f_1' \in J(f_1)} \cdots \sum_{f_r' \in J(f_r)}
        \phi(f_1' \cdots f_r') \otimes \Lambda(\phi(f_1 \setminus f_1') \otimes \dots \otimes
        \phi(f_r \setminus f_r')) \\
        &= \phi(t) \otimes 1 + \sum_{f_1' \in J(f_1)} \cdots \sum_{f_r' \in J(f_r)} \phi(f_1'
        \cdots f_r') \otimes \phi(\tilde B_+((f_1 \setminus f_1') \otimes \dots \otimes (f_r \setminus
        f_r')) \\
        &= \phi(t) \otimes 1 + \sum_{f \in J(t)} \phi(f) \otimes \phi(t \setminus f) \\
        &= (\phi \otimes \phi)(\Delta t)
    \end{align*}
    as desired.
\end{proof}

\begin{example} \label{example:boring universal map}
    Consider the (boring) 1-cocycle $B_+\mu \in \Cycles^1(\CK, \CK^{\otimes r})$. By \cref{thm:ck
    universal multilinear 1}, there is a unique $\beta\colon \CKE_r \to \CK$ such that $\beta \tilde
    B_+ = B_+\mu\beta^{\otimes r} = B_+\beta\mu$. It is easily seen that such a map is given by simply
    forgetting the edge decorations. More generally, consider any boring 1-cocycle $\Lambda\mu$ on
    any bialgebra $H$. Let $\psi\colon \CK \to H$ be the unique map such that $\psi B_+ = \Lambda
    \psi$ and let $\beta\colon \CKE_r \to \CK$ continue to denote the edge-undecorating map. Now we
    observe that
    \[
        \psi \beta \tilde B_+ = \psi B_+ \mu\beta^{\otimes r} = \Lambda \psi^{\otimes r} \mu
        \beta^{\otimes r} = \Lambda \mu (\psi \beta)^{\otimes r}
    \]
    so the universal map is simply $\phi = \psi\beta$ and the boring case is indeed boring.
\end{example}

In the next section we will generalize the tubing expansion to the map from \cref{thm:ck universal
multilinear 1}. As before, we will want a more general version including several 1-cocycles at once.
For this it is convenient to allow tensor products indexed by arbitrary finite sets rather than just
ordinals. Note that everything we did with $\KK[L]$ still works here: we can identify
$\KK[L]^{\otimes E}$ with $\KK[L_e\colon e \in E]$, and all 1-cocycles are integro-differential
operators as in \cref{thm:polynomial cocycles tensor} but slightly more notationally challenging.

Let $I$ be a set and $\mathcal E = \{E_i\}_{i \in I}$ be a family of finite sets.  Define a
\textit{$(I, \mathcal E)$-tree} to be a rooted tree where each vertex is decorated by an element of
$I$ and each edge from a parent of type $i$ is decorated by an element of $E_i$. Denote the set of
$(I, \mathcal E)$-trees by $\widetilde{\mathcal T}(I, \mathcal E)$. Analogously we have $(I,
\mathcal E)$-forests and we denote the set of these by $\widetilde{\mathcal F}(I, \mathcal E)$. Let
$\CKE_{I,\mathcal E}$ denote the free vector space on $\widetilde{\mathcal F}(I, \mathcal E)$.  As
usual, we make this into a bialgebra with disjoint union as the product and a downset/upset
coproduct preserving the decorations on the vertices and (remaining) edges. For $i \in I$, let
$\tilde B_+^{(i)}\colon \CKE_{I, \mathcal E}^{\otimes E_i} \to \CKE_{I, \mathcal E}$ be the operator
that adds a new root joined to each component with the appropriate decoration. These are 1-cocycles
by the same argument as in \cref{thm:tildebplus is 1-cocycle}.

\begin{theorem} \label{thm:ck universal multilinear 2}
    Let $A$ be a commutative algebra and $\{\Lambda_i\}_{i \in I}$ a family of linear maps,
    $\Lambda_i\colon A^{\otimes E_i} \to A$. There exists a unique algebra morphism $\phi\colon
    \CKE_{I, \mathcal E} \to A$ such that $\phi B_+^{(i)} = \Lambda_i \phi^{\otimes E_i}$.
    Moreover, if $\Lambda_i$ is a 1-cocycle for each $i$ then $\phi$ is a bialgebra morphism.
\end{theorem}

\begin{proof}
    Analogous to \cref{thm:ck universal multilinear 1}.
\end{proof}

\begin{remark} \label{remark:boring cocycles are really boring}
    Analogously to \cref{example:boring universal map}, if \textit{all} of the 1-cocycles are boring
    there will be a factorization of $\phi$ through the map $\CKE_{I,\mathcal E} \to \CK_I$ that
    forgets edge decorations. However, we can make a more refined statement: if $\Lambda_i$ is
    boring, then $\phi$ is independent of the edge decorations for vertices of type $i$. Proving
    this is left as an exercise (though in the case that the target algebra is $\KK[L]$ it will
    follow from the results of the next section).
\end{remark}

\subsection{The renormalization group equation and the invariant charge}\label{subsec rge multiple}

The relationship between the renormalization group equation, the invariant charge and the Riordan group generalizes to the case of distinguished insertion places. Naturally, this will come from a generalization of \cref{thm:cocycle fdb
implies rio} to allow cocycles on tensor powers of the target bialgebra $H$. In this case the
statement gets more complicated but the proof is much the same. First we need a generalization of
\cref{thm:general rio coproduct}.

\begin{lemma} \label{thm:general rio coaction}
    Let $\delta$ be the left coaction of $\RioHopf$ on $\RioHopf^{\otimes E}$ for $E$ a finite set.
    Then for any $\mathbf u \in \KK^r$ and any exponent vector $\alpha \in \Nat^E$,
    \[
        \delta\left(\bigotimes_{e \in E} Y(x)^{u_e} \Pi(x)^{\alpha_e} \right) =
        \sum_{j \ge 0} Y(x)^{|\mathbf u|} \Pi(x)^{j} \otimes [x^j] \bigotimes_{e \in E}
        Y(x)^{u_e} \Pi(x)^{\alpha_e}.
    \]
\end{lemma}

\begin{proof}
    Immediate from \cref{thm:general rio coproduct} by the definition of the coaction.
\end{proof}

With this we can prove the following.

\begin{lemma} \label{thm:cocycle tensor fdb implies rio}
    Let $H$ be a bialgebra, $P$ be a set, and $\{E_p\}_{p \in P}$ a family of finite sets. For $p \in
    P$ let $\Lambda_p\colon H^{\otimes E_p} \to H$ be a 1-cocycle, let $\mathbf u_p = (u_e)_{e \in
    E_p}$ be a vector with $|\mathbf u| = 1$, and let $\mathbf w_p = (w_e)_{e \in E_p}$ be a
    nonzero exponent vector. Suppose $\phi\colon \RioHopf \to H$ is an algebra morphism. Let
    $\Phi(x) = \phi(\Pi(x))$ and suppose $\phi(Y(x)) = F(x)$ where $F(x)$ is the unique solution to
    \[
        F(x) = 1 + \sum_{p \in P} \Lambda_p\left(\bigotimes_{e \in E_p} F(x)^{u_{e}}
        \Phi(x)^{\alpha_{e}}\right).
    \]
    Then for $n \ge 0$, if $\phi$ is a bialgebra morphism when restricted to $\FdB^{(n)}$, it is
    also a bialgebra morphism when restricted to $\RioHopf^{(n)}$.
\end{lemma}

\begin{proof}
    The setup for our induction argument is identical to that in the proof of \cref{thm:cocycle fdb
    implies rio}. Thus, supposing that $\phi$ is a bialgebra morphism on $\RioHopf^{(n-1)}$ for some
    $n \ge 1$, we set out to show that $\phi$ preserves the coproduct of $y_n$. Note that we have no
    $y_n$ in $[x^n] \bigotimes_{e \in E_p} Y(x)^{u_e} \Pi(x)^{w_e}$ since $|\mathbf w_p| > 0$. Thus by
    \cref{thm:general rio coaction} we have
    \begin{align*}
        \Delta \phi(y_n)
        &= \Delta([x^n]F(x)) \\
        &= [x^n] F(x) \otimes 1 + \sum_{p \in P} (\id \otimes \Lambda_p) [x^n]\delta\left(
        \bigotimes_{e \in E_p} F(x)^{u_{e}} \Phi(x)^{w_{e}}\right) \\
        &= [x^n] F(x) \otimes 1 + \sum_{p \in P} \sum_{j \ge 0} [x^n] F(x)\Phi(x)^j \otimes \Lambda_p
        \left([x^j] \bigotimes_{e \in E_p} F(x)^{u_{e}} \Phi(x)^{w_{e}}\right) \\
        &= \sum_{j \ge 0} [x^n] F(x)\Phi(x)^j \otimes [x^j] F(x)
    \end{align*}
    as desired.
\end{proof}

With \cref{thm:cocycle tensor fdb implies rio} in mind we can formulate an appropriate notion of
invariant charge for Dyson--Schwinger equations with distinguished insertion places. We want a
series $Q(x)$ to play the role of $\Phi(x)$ in the statement of the lemma. In the case of a single
equation \cref{eqn:multidse single} we see that the condition we want is that for some $s \in \KK$,
the insertion exponents can be written in the form
\begin{equation} \label{eqn:insertion exponent relation multidse single}
    \mu_{e} = u_e + sw_e
\end{equation}
where $u_e \in \KK$ and $\alpha_e \in \Nat$ are such that
\begin{equation} \label{eqn:insertion exponent relation for u}
    \sum_{e \in E_p} u_e = 1
\end{equation}
and
\begin{equation} \label{eqn:insertion exponent relation for alpha}
    \sum_{e \in E_p} w_e = w_p
\end{equation}
for each $p$. In this case we take $Q(x) = xT(x)^s$ as in the case of a single ordinary DSE. We can
then write our (combinatorial) equation as
\begin{equation} \label{thm:multicdse single with q}
    T(x) = \sum_{p \in P} \tilde B_+\left(\bigotimes_{e \in E_p} T(x)^{u_e} Q(x)^{\alpha_e}\right).
\end{equation}

\begin{remark}
    The choice of exactly how to write the insertion exponents in the form \cref{eqn:insertion
    exponent relation multidse single} is not unique in general. However the value of $s$ and hence
    of $Q(x)$ does not depend on this choice, because the \textit{overall} insertion exponents
    satisfy $\mu_p = 1 + sw_p$ regardless.

    Recall our original motivating example was equation \cref{eqn:original multidse} which has $m$
    edge types and all insertion exponents equal to $-1$. Thus we here have $s = -(1+m)$. We can
    write it in the form \cref{eqn:insertion exponent relation multidse single} by choosing one edge
    type $e_0$ to have $u_{e_0} = m$ and $w_{e_0} = 1$, and all others to have $u_e = -1$ and $w_e = 0$. Thus while writing this equations in the form \cref{thm:multicdse single with q} is
    convenient for our purposes in this section, it does involve arbitrarily breaking the symmetry
    of the original equation.
\end{remark}

For systems the situation is similar. For $e \in E_p$ where $p \in P_i$, we want the insertion
exponent vector to satisfy a relation
\[
    \mu_e = u_e1_i + \alpha_e \mathbf s
\]
where $u_e$ and $\alpha_e$ still satisfy \cref{eqn:insertion exponent relation for u} and
\cref{eqn:insertion exponent relation for alpha}. Thus the form of our system is
\begin{equation} \label{eqn:multicdse system with q}
    T_i(x) = \sum_{p \in P_i} \tilde B_+^{(p)}\left(\bigotimes_{e \in E_p} T_i(x)^{u_e}
    Q(x)^{\alpha_e}\right).
\end{equation}
where as before
\[
    Q(x) = x\prod_{i \in I} T_i(x)^{s_i}.
\]
With the setup done, we can state the main result of this section. This proves a conjecture of
Nabergall \cite[Conjecture 4.2.3]{nabergall:phd} which corresponds to the case all insertion
exponents equal $-1$.

\begin{theorem}
    Let $\mathbf T(x) \in \CKE_{P,\mathcal E}[[x]]^I$ be the solution to the combinatorial
    Dyson--Schwinger system \cref{eqn:multicdse system with q}. Then for any $i \in I$, the map
    $\phi_i\colon \RioHopf \to \CKE_{P,\mathcal E}$ defined by $\phi_i(Y(x)) = T_i(x)$ and
    $\phi_i(\Pi(x)) = Q(x)$ is a bialgebra morphism. As a consequence, the solution $\mathbf G(x, L)$
    to the corresponding Dyson--Schwinger system satisfies the
    renormalization group equations
    \[
        \left(\frac{\partial}{\partial L} - \beta(x) \frac{\partial}{\partial x} -
        \gamma_i(x)\right) G_i(x, L) = 0
    \]
    where $\gamma_i(x)$ is the linear term in $L$ of $G_i(x, L)$ and
    \[
        \beta(x) = \sum_{i \in I} s_ix\gamma(x).
    \]
\end{theorem}

\begin{proof}
    This follows by an identical proof to \cref{thm:dse system rio} but using \cref{thm:cocycle
    tensor fdb implies rio} in place of \cref{thm:cocycle fdb implies rio}.
\end{proof}

\subsection{Reducing to ordinary DSEs}\label{subsec qle}

Before we embark on generalizing the tubing expansion to Dyson--Schwinger systems with distinguished
insertion places, we will pause to consider an alternative approach, namely transforming such
systems into ordinary (single insertion place) Dyson--Schwinger systems. This is motivated by a conjecture of Nabergall
\cite[Conjecture 4.2.2]{nabergall:phd} that the solution to \cref{eqn:original multidse} can also be obtained by an explicit linear variable substitution from an also explicit single insertion place Dyson--Schwinger equation. This turns out to be false as demonstrated by the
following.

We will take the case $m=2$ in \cref{eqn:original multidse} and index the coefficients of the Mellin transform as
\[
(\rho_1+\rho_2)F(\rho_1, \rho_2) = \sum_{i,j\geq 0}b_{i,j}\rho_1^i\rho_2^j.
\]
Then, iterating \cref{eqn:original multidse} and calling the Green function in this case $\hat{G}$ we obtain
\begin{align*}
  \hat{G}(x,L) = \,& 1 + Lb_{0,0}x - (L^2b_{0,0}^2 + Lb_{0,0}(b_{0,1} + b_{1,0}))x^2 \\
             & + \bigg(\frac{5}{3}L^3b_{0,0}^3 + \frac{7}{2}L^2b_{0,0}^2(b_{0,1} + b_{1,0}) \\
             & \qquad + Lb_{0,0}(b_{0,1}^2 + 4b_{0,0}b_{0,2} + 2b_{0,1}b_{1,0} + b_{1,0}^2 + b_{0,0}b_{1,1} + 4b_{0,0}b_{2,0})\bigg)x^3 \\
             & - \bigg(\frac{10}{3}L^4b_{0,0}^4 + 11L^3b_{0,0}^3(b_{0,1} + b_{1,0}) \\
             & \qquad + L^2b_{0,0}^2\bigg(\frac{15}{2}b_{0,1}^2 + 20b_{0,0}b_{0,2} + 15b_{0,1}b_{1,0} + \frac{15}{2}b_{1,0}^2 + 5b_{0,0}b_{1,1} + 20b_{0,0}b_{2,0}\bigg) \\
             & \qquad + Lb_{0,0}(b_{0,1}^3 + 15b_{0,0}b_{0,1}b_{0,2} + 28b_{0,0}^2b_{0,3} + 3b_{0,1}^2b_{1,0} + 15b_{0,0}b_{0,2}b_{1,0} \\
             & \qquad \qquad + 3b_{0,1}b_{1,0}^2 + b_{1,0}^3 + 3b_{0,0}b_{0,1}b_{1,1} + 3b_{0,0}b_{1,0}b_{1,1} + 4b_{0,0}^2b_{1,2} \\
  & \qquad \qquad + 15b_{0,0}b_{0,1}b_{2,0} + 15b_{0,0}b_{1,0}b_{2,0} + 4b_{0,0}^2b_{2,1} + 28b_{0,0}^2b_{3,0})\bigg)x^4 + O(x^5).
\end{align*}
These calculations were done with SageMath, but at the cost of some tedium are also doable by hand.
The single insertion place Dyson--Schwinger equation would need to be \cref{eq simplest dse} with $\mu=-2$ so that the total exponent agrees.  Indexing the Mellin transform as after \cref{eq simplest dse} and calculating likewise we obtain
\begin{align*}
  G(x,L) = \,& 1 + La_0x - (L^2a_0^2 + 2La_0a_1)x^2 + (\frac{5}{3}L^3a_0^3 + 7L^2a_0^2a_1 + La_0(4a_1^2 + 10a_0a_2))x^3 \\
             & - \bigg(\frac{10}{3}L^4a_0^4 + 22L^3a_0^3a_1 + L^2a_0^2(30a_1^2 + 50a_0a_2) \\
  & \qquad + La_0(8a_1^3 + 72a_0a_1a_2 + 80a_0^2a_3)\bigg)x^4 + O(x^5).
\end{align*}
{}From the coefficients of $x$ we see that $a_0= b_{0,0}$.  Using that, from the coefficient of $x^2$ we see that $a_1= (b_{1,0}+b_{0,1})/2$, and likewise from the coefficient of $x^3$ we get $a_2 = (4b_{0,2}+b_{1,1}+4b_{2,0})/10$.

The problem comes with the coefficient of $x^4$.  Making the already established substitutions and taking the difference of the coefficient of $x^4$ in $\hat{G}$ and $G$ we obtain
\begin{align*}
  0 = \,& \frac{1}{5}(3b_{0,1}b_{0,2} + 140b_{0,0}b_{0,3} + 3b_{0,2}b_{1,0} - 3b_{0,1}b_{1,1} - 3b_{1,0}b_{1,1} + 20b_{0,0}b_{1,2} \\
  & \quad + 3b_{0,1}b_{2,0} + 3b_{1,0}b_{2,0} + 20b_{0,0}b_{2,1} + 140b_{0,0}b_{3,0} - 400b_{0,0}a_3)Lb_{0,0}^2.
\end{align*}
So we see that solving for $a_3$ would involve inverting $b_{0,0}$ and so no linear substitution into $G(x,L)$ can give $\hat{G}(x,L)$ and so in particular no substitution of the form in \cite[Conjecture 4.2.2]{nabergall:phd} can do so.  This disproves the conjecture.

Despite this negative result, there are some cases in which it \textit{is} possible to do this. One
example is when all the cocycles that appear in the system are boring: \cref{remark:boring
cocycles are really boring} essentially says that this is possible even at the level of
\textit{combinatorial} Dyson--Schwinger systems, see the discussion at the beginning of \cref{subsec tubing multiple} for further details.

There is one more case we know of in which such a transformation exists, namely when the overall
insertion exponents are all equal to 1, or equivalently $\beta(x) = 0$. In the case of
undistinguished insertion places, this would be a linear equation. For distinguished insertion
places we call such an equation \textit{quasi-linear}. Explicitly, quasi-linear DSEs are simply the
special case of \cref{eqn:multidse single} in which
\[
    \sum_{e \in E_p} \mu_e = 1
\]
for each $p$.

Quasi-linear DSEs turn out to inherit some of the special properties of linear DSEs. To see why, we
first consider the multiple-insertion-place analogue of \cref{eqn:differentiate dse}:
\[
    \frac{\partial}{\partial L} G(x, L) = \sum_{p \in P} x^{w_p} A_p\left(\frac{\partial}{\partial
    L_e} \colon e \in E_p\right) \prod_{e \in E_p} G(x, L_e)^{\mu_e}\bigg|_\text{all $L_e = L$}.
\]
For a \textit{general} DSE in multiple insertion places, any attempt to follow the template of
\cref{sec:equations for gamma} stops here. The RGE allows us to replace $\partial/\partial L_e$ with
a differential operator in $x$ acting on the factor for $e$ on the right, but to get an expression
like \cref{eqn:combine dse with rge} we would need an operator acting on the whole product. However,
in the quasi-linear case, we only get multiplication by a series with no differential part, so we
can press on to get
\[
    \frac{\partial}{\partial L} G(x, L) = \sum_{p \in P} x^{w_p} A_p\left(\mu_e\gamma(x) \colon e
    \in E_p\right) \prod_{e \in E_p} G(x, L)^{\mu_e}
\]
and then set $L = 0$ to get the analogue of \cref{eqn:equation for gamma functional}:
\begin{equation} \label{eqn:equation for gamma quasilinear}
    \gamma(x) = \sum_{p \in P} x^{w_p} A_p(\mu_e \gamma(x)\colon e \in E_p).
\end{equation}
But note that this functional equation is actually of the same form as \cref{eqn:equation for gamma
functional} and so in fact also arises from an ordinary linear DSE!

\begin{theorem} \label{thm:quasilinear}
    If $\sum_{e \in E_p} \mu_e = 1$ for each $p$, the solution to \cref{eqn:multidse single} is the
    same as the solution to the ordinary linear DSE
    \[
        G(x, L) = 1 + \sum_{p \in P} x^{w_p} \int_0^L \tilde A_p(\partial/\partial u) G(x, u)\,du
    \]
    where
    \[
        \tilde A_p(L) = A_p(\mu_e L\colon e \in E).
    \]
\end{theorem}

\begin{proof}
    By \cref{eqn:equation for gamma functional} and \cref{eqn:equation for gamma quasilinear} they
    have the same anomalous dimension, which is sufficient because of \cref{rmk rge meaning}. 
\end{proof}

\subsection{Tubing solutions to Dyson--Schwinger equations with multiple insertion places}\label{subsec tubing multiple}

Finally, our main result is
tubing expansions of solutions to Dyson--Schwinger equations with multiple insertion places.  Fix a
family $\mathcal E = \{E_i\}_{i \in I}$ of finite sets. For each $i$, introduce indeterminates
$\mathbf L_i = (L_e\colon e \in E_i)$ and choose a 1-cocycle $\Lambda_i \in \Cycles^1(\KK[L],
\KK[\mathbf L_i])$. Let $A_i(\mathbf L_i)$ be the corresponding power series given by
\cref{thm:polynomial cocycles tensor}. We will choose to expand the series using multinomial
coefficients:
\[
    A_i(\mathbf L_i) = \sum_{\alpha \in \Nat^{E_i}} a_{i,\alpha} \binom{|\alpha|}{\alpha} \mathbf
    L_i^\alpha.
\]
This curious-looking convention can be justified by the observation that if $\Lambda_i$ is boring
then, by \cref{remark:boring polynomial cocycles}, we can write
\[
    A_i(\mathbf L_i) = B\left(\sum_{e \in E_i} L_e\right)
\]
for some series $B(L)$. Our convention is such that in this case we have $a_{i,\alpha} =
[L^{|\alpha|}] B(L)$. In particular, this will make our expansion manifestly identical to
\cref{thm:tubing expansion 1} in the case that all of the cocycles are boring, and more generally
make it obviously independent of the edge decorations for those vertex types with boring 1-cocycles
as suggested by \cref{remark:boring cocycles are really boring}.

\begin{figure}
    \begin{center}
        \includegraphics[page=3,width=0.3\textwidth]{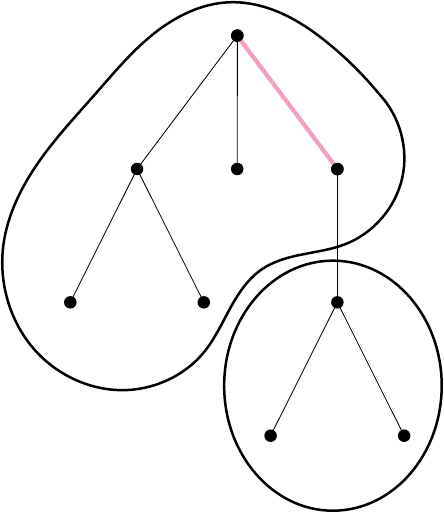}
    \end{center}
    \caption{An upper tube and its corresponding lower tube. The type of the upper tube is the
    decoration of the highlighted edge.}
\end{figure}

\label{symboldef:rankv} To generalize the tubing expansion to this case, we need the appropriate
generalizations of the statistics that appear in \cref{thm:tubing expansion 1}. Let $t$ be an $(I,
\mathcal E)$-tree and $\tau$ be a binary tubing of $t$. Suppose $t' \in \tau$ is an upper tube and
$t''$ the corresponding lower tube (i.e. the unique lower tube such that $t' \cup t'' \in \tau$). We
define the \textit{type} of $t'$ to be the decoration of the first edge on the unique path from $\rt
t'$ to $\rt t''$. By construction, the type is an
element of $E_i$ where $i = d(t')$ is the decoration of the root vertex of $t'$. Note we only assign
types to upper tubes, not lower tubes. For each vertex $v \in t$ and edge type $e \in E_{d(v)}$
define the \textit{$e$-rank} $\rank_e(\tau, v)$ to be the number of upper tubes of type $e$ rooted
at $v$.  Collect these together to get the \textit{rank vector} $\rankv(\tau, v) \in
\Nat^{E_{d(v)}}$. Clearly $|\rankv(\tau, v)| = \rank(\tau, v)$. \label{symboldef:mel2} Then we define
the \textit{Mellin monomial} of $\tau$ to be
\[
    \mel(\tau) = \prod_{\substack{v \in t \\ v \ne \rt t}} a_{d(v), \rankv(\tau, v)}.
\]

\label{symboldef:betavec}The analogue of the $b$-statistic is slightly more complicated. For $1 \le
k \le b(\tau)$ write $\beta_i^k(\tau)$ the number of upper tubes of type $i$ containing (hence
rooted at) $\rt t$, excluding the outermost $k - 1$ upper tubes. Collect these into a vector
$\beta^k(\tau)$ (which for lack of a better name we simply term the \textit{$k$th $\beta$-vector} of
$\tau$). Thus $\beta^1(\tau) = \rankv(\tau, \rt t)$ and $|\beta^k(\tau)| = b(\tau) - k$. We are now
ready to state our expansion.

\begin{theorem} \label{thm:tubing expansion 2}
    With the above setup, the unique map $\phi\colon \CKE_{I, \mathcal E} \to \KK[L]$ satisfying
    $\phi \tilde B_+^{(i)} = \Lambda_i \phi$ is given on trees by the formula
    \begin{equation} \label{eqn:tubing expansion 2}
        \phi(t) = \sum_{\tau \in \Tub(t)} \mel(\tau) \sum_{k=1}^{b(\tau)} a_{d(t),\beta^k(\tau)}
        \frac{L^k}{k!}.
    \end{equation}
\end{theorem}

To prove this, we follow the structure of the argument as in Section 4 of \cite{tubings}, with
some added complexity. We will temporarily write $\psi$ for the right side of
\cref{eqn:tubing expansion 2}. Let $\sigma$ be the linear term of $\psi$, i.e.
\[
    \sigma(t) = \sum_{\tau \in \Tub(t)}a_{d(t), \beta^1(\tau)} \mel(\tau) .
\]
and $\sigma$ vanishes on disconnected forests.

\begin{lemma} \label{thm:tubing tensor lemma 1}
    For any tree $t$ and $k \ge 1$,
    \[
        \sigma^{*k}(t) = \sum_{\substack{\tau \in \Tub(t) \\ b(\tau) \ge k}} a_{d(t), \beta^k(\tau)}
        \mel(\tau).
    \]
\end{lemma}

(Note that, $\sigma^{*k}$ can be nonzero on disconnected forests, but
we are claiming this equality only for trees (see also the remark after Lemma 4.3 in \cite{tubings}).)

\begin{proof}
    By induction on $k$. The base case is true by definition. Then
    \[
        \sigma^{*k+1}(t) = \sum_{f \in J(t)} \sigma(f) \sigma^{*k}(t \setminus f)
    \]
    but $\sigma$ is an infinitesimal character so $\sigma(f) \ne 0$ only if $f$ is actually a tree.
    Moreover, clearly $\sigma^{*k}(1) = 0$, so we can restrict the sum to proper subtrees $t'$. Then
    inductively we have
    \begin{align*}
        \sigma^{*k+1}(t)
        &= \sum_{t'} \sigma(t') \sigma^{*k}(t \setminus t') \\
        &= \sum_{t'} \sum_{\tau' \in \Tub(t')} \sum_{\substack{\tau'' \in \Tub(t \setminus t') \\
        b(\tau'') \ge k}} a_{d(t'), \beta^1(\tau')} \mel(\tau') a_{d(t),
        \beta^k(\tau'')} \mel(\tau'').
    \end{align*}
    By the recursive construction of tubings (\cref{thm:recursive tubing}), $\tau'$ and $\tau''$
    uniquely determine a tubing $\tau \in \Tub(t)$. Note that for any vertex $v \in t'$, the upper
    tubes of $\tau$ rooted at $v$ are the same as those of $\tau'$, so $\rankv(\tau, v) =
    \rankv(\tau', v)$. The same is true for vertices of $\tau''$ other than $\rt t$. Thus we have
    have
    \[
        \mel(\tau) = \mel(\tau')a_{d(t'), \rankv(\tau, \rt t')} \mel(\tau'') = \mel(\tau')
        a_{d(t'), \beta^1(\tau')} \mel(\tau'').
    \]
    Moreover, since $\beta^k$ ignores the outermost the $k - 1$ upper tubes containing $\rt t$ by
    definition, we have $\beta^{k+1}(\tau) = \beta^k(\tau'')$. Finally, there is one additional tube
    in $\tau$ containing the root, so $b(\tau) = b(\tau'') + 1$. Hence the triple sum simplifies to
    \[
        \sigma^{*k+1}(t) = \sum_{\substack{\tau \in \Tub(t) \\ b(\tau) \ge k+1}} a_{d(t),
        \beta^{k+1}(\tau)} \mel(\tau)
    \]
    as wanted.
\end{proof}

For notational convenience, for $\alpha \in \Nat^{E_i}$ let us write $\sigma^{[\alpha]}$ for the
linear form on $\CK_{D, \mathcal I}^{\otimes I_d}$ given by
\[
    \sigma^{[\alpha]} = \bigotimes_{e \in E_d} \sigma^{*\alpha_e}.
\]
Note that the coalgebra structure on $\CK_{I, \mathcal E}^{\otimes E_i}$ gives a convolution product
on linear forms; one easily checks that $\sigma^{[\alpha + \beta]} = \sigma^{[\alpha]} *
\sigma^{[\beta]}$.

\begin{lemma} \label{thm:tubing tensor lemma 2}
    For $i \in I$,
    \[
        \sigma \tilde B_+^{(i)} = \sum_{\alpha \in \Nat^{E_i}} \binom{|\alpha|}{\alpha} a_{i, \alpha}
        \sigma^{[\alpha]}.
    \]
\end{lemma}

\begin{proof}
    For each $i \in I$ and $\alpha \in \Nat^{E_i}$ let $\sigma_{i, \alpha}$ be the infinitesimal
    character defined by
    \[
        \sigma_{i, \alpha}(t) = \begin{cases}
            \sum_{\tau \in \Tub(t), \beta^1(\tau) = \alpha} \mel(\tau) & d(t) = i \\
            0 & \text{otherwise}
        \end{cases}
    \]
    so that we have
    \[
        \sigma = \sum_{i \in I} \sum_{\alpha \in \Nat^{E_i}} a_{i, \alpha} \sigma_{i, \alpha}.
    \]
    Note that we have $\sigma_{i, \alpha}\tilde B_+^{(j)} = 0$ when $i \ne j$, so to get the desired
    formula it suffices to show that $\sigma_{i, \alpha}\tilde B_+^{(i)} = \binom{|\alpha|}{\alpha}
    \sigma^{[\alpha]}$ for $i \in I$ and $\alpha \in \Nat^{E_i}$. We do this by induction on $m =
    |\alpha|$.

    Now, for $\alpha = 0$ we have that $\sigma_{i,\alpha}(t) = 1$ if $t$ is the one-vertex tree with
    decoration $i$ and 0 otherwise. Of course $\sigma^{[0]}$ is 1 when each forest is empty and 0
    otherwise, so we do indeed see that $\sigma_{i,0}\tilde B_+^{(i)} = \sigma^{[0]}$ as wanted.

    Suppose now that $m > 0$ and that the desired identity holds for smaller values. Then
    $\sigma_{i, \alpha}$ vanishes on one-vertex trees, so all tubings of interest have at least one
    upper tube. For $e \in E_i$, let $\sigma_{i,\alpha}^{e}(t)$ be the sum only for those tubings
    where the outermost upper tube has type $e$. Thus
    \[
        \sigma_{i, \alpha} = \sum_{e \in E_i} \sigma_{i,\alpha}^{e}.
    \]
    Note $\sigma^e_{i, \alpha} = 0$ when $\alpha_e = 0$, as in this case there must be no tube of
    type $i$ containing the root. Suppose now that $\alpha_e \ne 0$ and $t = \tilde
    B_+^{(i)}\left(\bigotimes_{e' \in E_i} f_{e'}\right)$. Let $\tau$ be a binary tubing of $t$
    which recursively corresponds to $(\tau', \tau'')$. Then the outermost upper tube of $\tau$ is
    type $e$ precisely when the lower subtree $t'$ is contained in $f_e$. Moreover in this case we
    have $\beta^1(\tau) = \beta^1(\tau'') + 1_e$, where $1_e$ is the indicator vector for $e$.
    Then we have
    \[
        \sigma^e_{i, \alpha}(t) = \sum_{\substack{t' \subseteq f_i \\ \text{subtree}}} \sigma(t')
        \sigma_{i, \alpha - 1_i}(t \setminus t').
    \]
    Now $t \setminus t' = \tilde B_+^{(i)}\left(\bigotimes_{e'} f'_{e'}\right)$ where $f'_e = f_e
    \setminus t'$ and $f'_{e'} = f_{e'}$ for $e \ne e'$. It follows that
    \begin{align*}
        \sigma^e_{i, \alpha}B_+^{(i)}
        &= \sigma^{[1_e]} * \sigma_{i, \alpha - 1_e} \tilde B_+^{(i)} \\
        &= \sigma^{[1_e]} * \binom{m-1}{\alpha - 1_e} \sigma^{[\alpha - 1_e]} \\
        &= \binom{m-1}{\alpha - 1_e} \sigma^{[\alpha]}.
    \end{align*}

    Finally, by summing over the values of $e$ we get
    \[
        \sigma_{i, \alpha} \tilde B_+^{(i)} = \sum_{e \in E_i} \binom{m - 1}{\alpha - 1_e}
        \sigma^{[\alpha]} = \binom{m}{\alpha} \sigma^{[\alpha]}
    \]
    by the multinomial analogue of the Pascal recurrence.
\end{proof}

\begin{proof}[Proof of \cref{thm:tubing expansion 2}]
    Let $\Psi_i\colon \CKE_{I, \mathcal E}^{\otimes E_i} \to \KK[\mathbf L_i]$ be given by
    \[
        \Psi_i\left(\bigotimes_{e \in E_i} f_e\right) = \prod_{e \in E_i}
        \left(\psi(f_e)\big|_{L = L_e}\right).
    \]
    This is simply the map $\psi^{\otimes E_i}$ carried through the identification of
    $\KK[L]^{\otimes E_i}$ with $\KK[\mathbf L_i]$. Thus, our goal is to show $\psi \tilde B_+^{(i)}
    = \Lambda_i \Psi_i$; by uniqueness this shows $\phi = \psi$. Note that since $\psi$ is an
    algebra morphism, we have
    \[
        \Psi_i\left(\bigotimes_{e \in E_i} f_e\right)\bigg|_{L_e = L\,\forall e \in E_i} = \prod_{e
        \in E_i} \psi(f_e) = \psi\left(\prod_{e \in E_i} f_e\right).
    \]

    By \cref{thm:tubing tensor lemma 1} we have $\psi = \exp_*(L\sigma)$. Thus
    \[
        \frac{d}{d L} \psi \tilde B_+^{(i)}
        = (\psi * \sigma) \tilde B_+^{(i)}
        = \psi *_\delta \sigma \tilde B_+^{(i)}
    \]
    where $\delta$ is the coaction and the second equality is by \cref{thm:cocycle convolution}. But
    observe that
    \begin{align*}
        \left(\psi *_\delta \sigma\tilde B_+^{(i)}\right)\left(\bigotimes_{e \in E_i} f_e\right)
        &= \sum_{\substack{f_e' \in J(f_e) \\ \forall e \in E_i}} \psi\left(\prod_{e \in E_i}
        f'_e\right) \sigma\left(\tilde B_+^{(i)}\left(\bigotimes_{e \in e_i} (f_e \setminus f'_e)
        \right)\right) \\
        &= \sum_{\substack{f_e' \in J(f_e) \\ \forall e \in E_i}} \Psi_i\left(\bigotimes_{e \in E_i}
        f'_e\right) \sigma\left(\tilde B_+^{(i)}\left(\bigotimes_{e \in e_i} (f_e \setminus f'_e)
        \right)\right) \bigg|_{L_e = L\,\forall e \in E_i} \\
        &= \left(\Psi_i * \sigma \tilde B_+^{(i)}\right)\left(\bigotimes_{e \in E_i}
        f_e\right)\bigg|_{L_e = L\,\forall e \in E_i}.
    \end{align*}

    Thus we have
    \begin{align*}
        \frac{d}{d L} \psi \tilde B_+^{(i)}
        &= \Psi_i * \sigma \tilde B_+^{(i)} \big|_{L_e = L\,\forall e \in E_i} \\
        &= \left(\Psi_i * \sum_{\alpha \in \Nat^{E_i}} \binom{|\alpha|}{\alpha} a_{i,\alpha}
        \sigma^{[\alpha]}\right)\bigg|_{L_e = L\,\forall e \in E_i} & \text{by \cref{thm:tubing
        tensor lemma 2}} \\
        &= A_i(\partial/\partial L_e\colon e \in E_i) \Psi\big|_{L_e = L\,\forall e \in E_i} \\
        &= \frac{d}{dL} \Lambda_i\Psi.
    \end{align*}
    Since we also have
    \[
        \psi(\tilde B_+^{(i)} 1) = a_{i,0} = \Lambda_i(1)
    \]
    this implies $\psi \tilde B_+^{(i)} = \Lambda_i \Psi_i$, as wanted.
\end{proof}

Finally, to solve \cref{eqn:multidse single}
and \cref{eqn:multidse system}, we need to lift them to combinatorial versions on the Hopf algebra $\CKE_{P, \mathcal E}$
we introduced above, solve those, and then apply \cref{thm:tubing expansion 2} to
get a solution to the original equations. This time we will work in the full generality of systems
from the start. The combinatorial version of the system \cref{eqn:multidse system} is
\begin{equation} \label{eqn:multicdse system}
    T_i(x) = 1 + \sum_{p \in P_i} x^{w_p} \tilde B_+^{(p)}\left(\bigotimes_{e \in E_p} \mathbf
    T(x)^{\mu_e}\right).
\end{equation}
As in \cref{subsec dse}, we are slightly abusing notation here by neglecting to notate the obvious (but non-injective) map
$\CKE_{P,\mathcal E}[[x]]^{\otimes E_p} \to \CKE_{P,\mathcal E}^{\otimes E_p}[[x]]$. The formula
generalizes that of \cref{thm:cdse system solution}. Each vertex will now have several outdegree
vectors, one for each edge type: we write $\od_i(v, e)$ for the number of children of $v$ which have
decorations lying in $P_i$ and such that the edge connecting them has decoration $e$. These are
collected together into the outdegree vector $\odv(v, e) \in \Nat^I$.

\begin{theorem} \label{thm:multicdse system solution}
    The unique solution to \cref{eqn:multicdse system} is
    \[
        T_i(x) = 1 + \sum_{t \in \Trees(P_i)} \left(\prod_{v \in t} \prod_{e \in E_{d(v)}}
        \mu_e^{\underline{\odv(v, e)}}\right) \frac{tx^{w(t)}}{|\Aut(t)|}.
    \]
\end{theorem}

\begin{proof}
    Analogous to \cref{thm:cdse single solution}.
\end{proof}

As before, we can get a combinatorial expansion for the solution of the Dyson--Schwinger equation by
applying our results on 1-cocycles.

\begin{theorem} \label{thm:multidse system solution}
    The unique solution to \cref{eqn:multidse system} is
    \[
        G_i(x, L) = 1 + \sum_{t \in \Trees(P_i)} \left(\prod_{v \in t} \prod_{e \in E_{d(v)}}
        \mu_e^{\underline{\odv(v, e)}}\right)\sum_{\tau \in \Tub(t)} \mel(\tau) \sum_{k=1}^{b(\tau)}
        a_{d(t),\beta^k(\tau)} \frac{x^{w(t)}L^k}{|\Aut(t)|k!}.
    \]
\end{theorem}

\begin{proof}
    Immediate from \cref{thm:multicdse system solution} and \cref{thm:tubing expansion 2}.
\end{proof}

\section{Conclusion}\label{sec conclusion}

We have given a method to give combinatorially indexed and understandable series solutions to Dyson--Schwinger equations and systems thereof with distinguished insertion places.  This concludes the single scale portion of the long-standing program one of us has to solve Dyson--Schwinger equations combinatorially.  These solutions are powerful in that they give us combinatorial control over the expansions.  So far this has resulted in physically interesting results on the leading log hierarchy \cite{cy17, cy20} and on resurgence \cite{bdy}, both of which would be interesting to investigate in the multiple insertion place case.

Paul Balduf has done some numerical computations of solutions to Dyson--Schwinger equations with two
insertion places.  See pp. 44, 45 of \cite{Bslides} where he plots the two parameters controlling the exponential growth rate of the coefficients of the anomalous dimension (what he calls $\mu$ and $\lambda$ in the form $(-\lambda)^{-n}\Gamma(n-\mu)$) as a function of the power of the insertions (in our notation $\mu_1-1$ and $\mu_2-1$).  Interestingly, when the two insertion places
degenerate into one by setting the power of one of the insertions to $-1$ (that is, one of the $\mu_i = 0$), then
the function does not appear to be differentiable based on the result of the numerical computation.
It is not clear why this should be, and we certainly have no combinatorial understanding in this
direction at present.

The program of solving Dyson--Schwinger equations with combinatorial expansions has also been interesting for pure combinatorics, first contributing to a renaissance in chord diagram combinatorics, and now suggesting connections to other notions of tubings in combinatorics, eg \cite{carr-devadoss, galashin:tubings}, which deserve further investigation.

The way in which trees with edge decorations and their Hopf algebraic structures appear in the study of regularity structures, as described in \cite{BHZregularity}, has a very similar flavour to what we see in our work on Dyson--Schwinger equations.  It would be very interesting to investigate a more precise connection between these areas.

The main case of Dyson--Schwinger equation still outstanding is the case of non-single scale insertions as occurs in general vertex insertions.  Non-single scale insertions can be addressed in the present framework by choosing one scale and then capturing the remaining kinematic dependence in dimensionless parameters, however this is inelegant in requiring a choice, and does not give any combinatorial or algebraic insight into the interplay of the different scales.

We have also used the Riordan group to more clearly explain the connection between Dyson--Schwinger equations and the renormalization group equation clarifying the known results in the single insertion case and generalizing to the multiple insertion place case.  We proved and disproved conjectures of Nabergall \cite{nabergall:phd}.

\printbibliography

\end{document}